\documentclass[12pt]{amsart}
\usepackage{latexsym}
\usepackage{amssymb, amsmath, amsthm}
\usepackage{color}
\usepackage{mathabx}
\usepackage[shortlabels]{enumitem}
\usepackage{graphicx}

\usepackage{geometry}
\usepackage[bookmarks=false]{hyperref}

\newtheorem{thm}{Theorem}
\newtheorem*{thm*}{Theorem}
\newtheorem*{question*}{Question}
\newtheorem{prop}{Proposition}[section]
\newtheorem{lem}[prop]{Lemma}

\newtheorem{cor}[prop]{Corollary}

\newtheorem{rem}[prop]{Remark}

\DeclareMathOperator{\bR}{\mathbb{R}}

\DeclareMathOperator{\bZ}{\mathbb{Z}}

\DeclareMathOperator{\cA}{\mathcal{A}}

\DeclareMathOperator{\cC}{\mathcal{C}}
\DeclareMathOperator{\cD}{\mathcal{D}}
\DeclareMathOperator{\cE}{\mathcal{E}}
\DeclareMathOperator{\cF}{\mathcal{F}}

\DeclareMathOperator{\cI}{\mathcal{I}}

\DeclareMathOperator{\cK}{\mathcal{K}}
\DeclareMathOperator{\cL}{\mathcal{L}}
\DeclareMathOperator{\cM}{\mathcal{M}}

\DeclareMathOperator{\cP}{\mathcal{P}}

\DeclareMathOperator{\cS}{\mathcal{S}}
\DeclareMathOperator{\cT}{\mathcal{T}}

\DeclareMathOperator{\cX}{\mathcal{X}}
\DeclareMathOperator{\cY}{\mathcal{Y}}
\DeclareMathOperator{\cZ}{\mathcal{Z}}

\DeclareMathOperator{\ft}{\mathfrak{t}}

\DeclareMathOperator{\PFC}{\operatorname{PFC}}
\DeclareMathOperator{\PFH}{\operatorname{PFH}}

\DeclareMathOperator{\ECH}{\operatorname{ECH}}
\DeclareMathOperator{\ECC}{\operatorname{ECC}}
\DeclareMathOperator{\CZ}{\operatorname{CZ}}

\title[Low-action holomorphic curves and invariant sets]{Low-action holomorphic curves and invariant sets}

\author{Dan Cristofaro-Gardiner \and Rohil Prasad}
\email{dcristof@umd.edu, rrprasad@berkeley.edu}

\begin{document}

\begin{abstract}

We prove a compactness theorem for sequences of low-action punctured holomorphic curves of controlled topology, in any dimension, without imposing the typical assumption of uniformly bounded Hofer energy.  In the limit, we extract a family of closed Reeb-invariant subsets.   Then, we prove new structural results for the $U$-map in ECH and PFH, implying that such sequences exist in abundance in low-dimensional symplectic dynamics.

We obtain applications to symplectic dynamics and to the geometry of surfaces.  First, we prove  generalizations to higher genus surfaces and three-manifolds of the celebrated Le Calvez--Yoccoz theorem.   Second, we show that for any closed Riemannian or Finsler surface a dense set of points have geodesics passing through them that visit different sections of the surface.  Third, we prove a version of Ginzburg--G\"{u}rel's ``crossing energy bound'' for punctured holomorphic curves, of arbitrary topology, in symplectizations of any dimension.  
\end{abstract}

\maketitle

\tableofcontents

\section{Introduction}

\subsection{Main results}

The detection and classification of compact invariant sets is a fundamental question in dynamical systems.  For example, a natural question asks how much of the parameter space is seen by a given trajectory; the closure of any such trajectory is a compact invariant set.  For diffeomorphisms of the circle and smooth flows on the plane, the theorems of Denjoy \cite{Denjoy32} and Poincar\'e--Bendixson \cite{Bendixson01} give a nearly complete picture.  The situation in higher dimensions, on the other hand, is much more mysterious, with a far greater diversity in the possible behaviors.  

 It is the detection question for invariant sets that will concern us here.  Perhaps the main theme of our paper is that, at least in low dimensional symplectic dynamics, such sets exist in abundance.   To prove this, we prove a new result, that holds in any dimension and for a broad class of dynamical systems, detecting invariant sets via Gromov's theory of pseudoholomorphic curves; this theorem is inspired by the recent work of Fish--Hofer \cite{FH23}, producing invariant sets for Hamiltonian flows on $\mathbb{R}^4$.  
 
What seems special about low-dimensions, at least from the point of view of the present work, is that, as we prove here, the necessary curves exist in a wide range of situations of interest.

We illustrate all this by first discussing the main dynamical applications.

\subsubsection{A general Le Calvez--Yoccoz property}
\label{sec:ley}
 
Recall that an invariant set $U$ of a map or flow, closed or not, is called \emph{minimal} if the orbit of each initial condition $p \in U$ is dense in $U$. A groundbreaking $1997$ paper by Le Calvez--Yoccoz \cite{LCY97}, improving on an earlier result of Handel \cite{Handel92}, showed that for any homeomorphism of $S^2$ the complement of an invariant finite set of points is never minimal.  Their result resolved the $2$-dimensional case of an old question of Ulam from the Scottish Book \cite[p. $208$]{scottish}.   

Our first results give a generalization of this, in the smooth symplectic setting, to higher-genus surfaces and $3$-manifolds. 

\begin{thm}\label{thm:2d_mid}
Let $\Sigma$ be a closed, oriented surface and let $\phi: \Sigma \to \Sigma$ be any monotone area-preserving diffeomorphism. Then for any proper compact invariant set $\Lambda \subset \Sigma$, the complement $\Sigma\,\setminus\,\Lambda$ is not minimal. 
\end{thm}

\begin{thm}\label{thm:3d_mid}
Let $Y$ be a closed, oriented $3$-manifold equipped with a co-oriented contact structure $\xi$ with torsion first Chern class. Let $\lambda$ be any contact form defining $\xi$ and let $\phi = \{\phi^t\}_{t \in \bR}$ denote the Reeb flow of $\lambda$. Then for any proper compact invariant set $\Lambda \subset Y$, the complement $Y\,\setminus\,\Lambda$ is not minimal.
\end{thm}

To put this in context, let us emphasize two main points.  The first involves the assumptions in the theorem.  Without any assumptions, the above theorems are clearly false, even in the conservative setting; one can consider, for example, an irrational translation or flow on a torus.  At the same time, the assumptions we impose apply to broad classes of systems in Hamiltonian dynamics.  For example, any Hamiltonian diffeomorphism of a closed symplectic surface is monotone, as is any rational area-preserving diffeomorphism of the $2$-torus.  As we will see, any geodesic flow on a closed Finsler surface corresponds to a Reeb flow with torsion first Chern class.  ``Most" three-manifolds are rational homology spheres, and these also give examples.  We discuss to what degree these assumptions could potentially be weakened in \S\ref{sec:rational}.

The next point involves what is novel about the results.
What is important 
in the context of conservative dynamics is the level of generality.  The simplest kinds of compact invariant sets are {\em periodic orbits}: this means that the invariant set is homeomorphic to a circle (in the case of flows) or a finite set of points with a transitive action (in the discrete setting). Previous results have established theorems like the above under strong dynamical assumptions such as the existence of only finitely many periodic orbits \cite{GG18, CGGM23};  prior results on the closing lemma (see \cite{Irie15, AsaokaIrie16, CGPZ21, EH21}) show that a $C^\infty$-generic system of the type we consider has a dense set of periodic points. 
That these theorems hold much more generally, at least in low dimensions, seems to us to be a quite new and arguably unexpected phenomenon.  

Theorem~\ref{thm:2d_mid} and Theorem~\ref{thm:3d_mid} guarantee the existence of an abundance of compact invariant sets. Moreover, the compact invariant sets are spread out in the manifold. For example, we obtain the following corollaries. 

\begin{cor}\label{cor:2d_elementary}
Under the assumptions of Theorem~\ref{thm:2d_mid}, the map $\phi$ has infinitely many distinct proper compact invariant sets whose union is dense in $\Sigma$. 
\end{cor}

\begin{cor}\label{cor:3d_mid}
Under the assumptions of Theorem~\ref{thm:3d_mid}, the Reeb flow has infinitely many distinct proper compact invariant sets whose union is dense in $Y$. 
\end{cor}

Corollary~\ref{cor:2d_elementary} and \ref{cor:3d_mid} are clearly false if one requires the compact invariant sets to be periodic orbits. For example, an irrational rotation of a two-sphere satisfies the assumptions of Theorem~\ref{thm:2d_mid}, but has just two periodic points.  

\subsubsection{The geometry of surfaces} 

In the setting of Riemannian or more generally Finslerian manifolds, it is natural to ask how much of the manifold is visited by a given geodesic. Perhaps the simplest dichotomy in this direction is between the dense and non-dense geodesics.  At one extreme, one could imagine that every geodesic is dense; this can not occur, because a closed geodesic can not be dense, but it is natural to wonder how far off it is from the actual behavior. In fact, for surfaces we have the following result contrasting this sharply:
 
\begin{thm}\label{thm:geodesics}
Let $F$ be any closed Finsler surface. Then there exists an infinite collection $G$ of geodesics such that
\begin{enumerate}[(a)]
\item $\gamma(\bR)$ is not dense in $F$ for any $\gamma \in G$;
\item The union $\bigcup_{\gamma \in G} \gamma(\bR)$ is dense in $F$;
\item The closures $\overline{\gamma(\bR)}$ and $\overline{\gamma'(\bR)}$ are distinct for any pair of distinct elements $\gamma, \gamma' \in G$. 
\end{enumerate}
\end{thm}

In other words, a dense set of points have a non-dense geodesic going through them; and, moreover, these geodesics all visit different sections of the surface.

In the negative curvature case, or the generic case, the above theorem is well-known, since in fact the stronger statement holds that the closed geodesics are dense.  However, we make no curvature assumption at all, and in this generality this property seems to be quite new.  Moreover, the theorem is false if one requires the geodesics to be closed. For example, there exist Finsler metrics on $S^2$ with exactly two geometrically distinct closed geodesics \cite{AK70}.

\subsubsection{Invariant sets from low-action holomorphic curves}

We now explain the promised detection theorem that we prove,
applicable in any dimension and of independent interest, for extracting invariant sets; it is used to prove all of the above dynamical results; it does not require a contact structure or even a stable Hamiltonian one.  


We work in the general setting of punctured holomorphic curves in symplectizations $\bR \times Y$ over framed Hamiltonian manifolds. A \emph{framed Hamiltonian structure} on a smooth, oriented manifold $Y$ of dimension $2n + 1 \geq 3$ is a pair $\eta = (\lambda, \omega)$ of a $1$-form $\lambda$ and a closed $2$-form $\omega$ such that $\lambda \wedge \omega^n > 0$. The \emph{Hamiltonian vector field} $R_\eta$ is defined implicitly by
$$\lambda(R_\eta) \equiv 1,\qquad \omega(R_\eta, -) \equiv 0.$$
The flow of $R_\eta$ preserves $\omega$ and the volume form $\lambda \wedge \omega^n$. This setup is an abstraction of many important classes of systems in symplectic and conservative dynamics, including mapping torii of symplectic diffeomorphisms, Reeb and stable Hamiltonian flows, and volume-preserving flows on three-manifolds. For example, if $\omega = d\lambda$, then $\lambda$ is a contact form and $R_\eta$ is its Reeb vector field. 

We follow the classical setup of holomorphic curve theory in symplectizations introduced by Hofer. Fix a Riemann surface $(C, j)$. A \emph{$J$-holomorphic curve} is a proper smooth map $u: C \to \bR \times Y$ satisfying the Cauchy--Riemann equation
$$J \circ Du = Du \circ j$$
 where $J$ is an \emph{$\eta$-adapted} almost-complex structure on $\bR \times Y$. This is a translation-invariant almost-complex structure restricting to a compatible almost-complex structure on the symplectic bundle $(\ker(\lambda), \omega)$ and sending $-R_\eta$ to the vector field $\partial_a$ defined by the $\bR$-coordinate on $\bR \times Y$. 
We say $u$ is \emph{standard} if the domain $C$ is homeomorphic to the complement of a finite subset of a closed Riemann surface. The geometry of a $J$-holomorphic curve in $\bR \times Y$ is controlled by the \emph{action} and \emph{Hofer energy}\footnote{This is not Hofer's original definition. Finiteness of $\mathcal{E}(C)$, however, is equivalent to finiteness of the original Hofer energy.}, defined respectively as
$$\mathcal{A}(u) := \int_C u^*\omega,\qquad \mathcal{E}(u) := \sup_{s \in \bR} \int_{C \cap u^{-1}(\{s\} \times Y)} u^*\lambda.$$
The action controls how far on average the tangent planes of $C$, which are $J$-invariant, are from the vertical plane spanned by $\partial_a$ and $R_\eta$. Therefore, a low-action holomorphic curve should approximate the vector field $R_\eta$ very well. The Hofer energy is, informally, the maximum length of the level sets of $C$ in $\bR \times Y$.


A key object in our method is the ``limit set'' of a sequence of holomorphic curves in a symplectization, which we now introduce. For any closed, smooth, odd-dimensional manifold $Y$, write $\cD(Y)$ for the space of pairs $(\eta, J)$ where $\eta$ is a framed Hamiltonian structure and $J$ is an $\eta$-adapted almost-complex structure. Equip it with the topology of $C^\infty$ convergence in both $\eta$ and $J$.  
Fix a pair $(\eta, J) \in \cD(Y)$ and a sequence $\{(\eta_k, J_k)\}$ in $\cD(Y)$ converging to $(\eta, J)$. Fix a sequence $\{u_k: C_k \to \bR \times Y\}$ where $u_k$ is $J_k$-holomorphic for each $k$. It is convenient to define $X := (-1, 1) \times Y$, and to define for any $s \in \bR$ the shift map $\tau_s: \bR \times Y \to \bR \times Y$, mapping $(t, y) \mapsto (t - s, y)$. Define the \emph{limit set} $\cX$ of the sequence $\{u_k\}$ to be the collection of all closed subsets $K \subseteq (-1, 1) \times Y$ arising as subsequential Hausdorff limits as $k \to \infty$ of height-$2$ slices of $u_k$. That is, there exists a sequence $\{s_k\}$ of real numbers such that 
a subsequence of 
$$\tau_{s_k} \cdot \Big(u_k(C_k) \cap (s_k - 1, s_k + 1) \times Y\Big) \subseteq X$$
converges in the Hausdorff topology to $K$. The limit set $\cX$ is a subset of $\cK(X)$, the space of all closed subsets of $X$ equipped with the topology of Hausdorff convergence. See \S~\ref{subsec:hausdorff} for a definition of the Hausdorff topology. 

The limit set has the following very important connectivity property:

\begin{prop} \label{prop:lim}
Fix a closed, smooth, oriented, odd-dimensional manifold $Y$ and fix a sequence $\{(\eta_k, J_k)\}$ converging in $\cD(Y)$ to a pair $(\eta, J)$. Let $\{u_k: C_k \to Y\}$ denote a sequence where $u_k$ is a standard $J_k$-holomorphic curve for each $k$. Then there exists a subsequence $\{u_{k_j}\}$ whose limit set $\cX$ is connected with respect to the Hausdorff topology. 
\end{prop}

A harder theorem, which we prove, is the following:

\begin{thm} \label{thm:limit_set_intro}
Fix a closed, smooth, oriented, odd-dimensional manifold $Y$ and a sequence $\{(\eta_k, J_k)\}$ converging in $\cD(Y)$ to a pair $(\eta, J)$. Let $\{u_k: C_k \to Y\}$ denote a sequence where $u_k$ is a standard $J_k$-holomorphic curve for each $k$ and let $\cX \subseteq \cK(X)$ denote their limit set.  Assume in addition that 
$$\lim_{k \to \infty} \cA(u_k) = 0\quad\text{and}\quad \inf_{k} \chi(C_k) > -\infty$$.
Then every set 
$\overline{\Lambda} \in \cX$ is equal to $(-1, 1) \times \Lambda$, where $\Lambda \in \cK(Y)$ is non-empty and invariant under the flow of the Hamiltonian vector field $R_\eta$. 
\end{thm}

The novelty\footnote{Indeed, in the case where $\sup_{k \geq 1} \cE(u_k) \leq E$ for some finite $E$, Theorem~\ref{thm:limit_set_intro} follows from the original work of Hofer, and moreover one obtains the stronger conclusion that any $\overline{\Lambda} \in \cX$ is a cylinder over a finite union of periodic orbits.}  of Theorem~\ref{thm:limit_set_intro} is that it extracts invariant sets without requiring that the Hofer energies $\{\cE(u_k)\}$ admit a finite $k$-independent upper bound, or even that any of the Hofer energies $\cE(u_k)$ are finite.  Bounds on Hofer energy are a standard assumption in the vast majority of the symplectic field theory literature; the only exceptions known to us are \cite{FH23, Prasad23} about ``feral" curves\footnote{Loosely speaking, these are curves with unbounded Hofer energy; they do not play a role in our arguments.}.  Our proof of Theorem~\ref{thm:limit_set_intro} is inspired by, and builds on, ideas in these works, though one should emphasize that the setting here, of a sequence of standard curves with action tending to zero, is quite different.  As we will see in the proof, this leads to new topological challenges, with an essential point being that the curves we accommodate can have genus and many punctures.  Indeed, it is essential for our arguments to be able to allow such curves.

\begin{rem}
\normalfont
Fish--Hofer in \cite[Definition 4.46]{FH23} also define a kind of limit set (the ``x-limit set"). Our notion of limit set differs from theirs in several ways which are crucial for our arguments. In particular,  \cite[Definition 4.46]{FH23} has an asymptotic condition, which we do not want to require for our applications. And, their limit set is a single invariant set, while ours is a connected family of invariant sets. This distinction is actually a key point, since the connectedness is exploited to show that such families contain many distinct invariant sets. 
\end{rem}

\subsubsection{Low action curves of controlled topology}

Of course, Theorem~\ref{thm:limit_set_intro} is only useful if the necessary holomorphic curves exist.  It has been known for some time \cite{CGHP19} that the ``$U$-map on ``embedded contact homology" can be used to produce low-action holomorphic curves, for any (nondegenerate) contact form on a closed three-manifold.  There is a parallel result for area-preserving surface diffeomorphisms, via periodic Floer homology \cite{CGPZ21, CGPPZ21}.  Another feature of our work, of independent interest, is a collection of results guaranteeing the desired topological control on these curves.    In fact, while for our purposes any lower bound on the Euler characteristic suffices, we show that ``most" of the curves have Euler characteristic bounded by $-2$; we defer the precise statement to \S\ref{sec:existcurves}.  It is here that the assumptions in our theorems in \S\ref{sec:ley} are used.

\subsection{The crossing energy theorem in symplectizations}

Our detection theorem has another application to ``crossing energy" theorems", which we now recall. The crossing energy theorem is a powerful tool introduced by Ginzburg--G\"{u}rel \cite{GG14}.  Recall that a neighborhood $U$ of a locally maximal compact invariant set $\Lambda$ of a homeomorphism or flow is \emph{isolating} if any compact invariant set $\Lambda' \subset U$ is a subset of $\Lambda$. The crossing energy theorem asserts that if $\Lambda$ is a locally maximal invariant set of a Hamiltonian diffeomorphism (for example, a hyperbolic fixed point), $U$ is an isolating neighborhood, and $V \subset U$ is such that $\overline{V} \subset U$, then any ``Floer cylinder" crossing the shell $U \setminus V$ must have a uniform lower bound on its Floer energy. Analogues have been established for generating functions \cite{Allais22}, gradient flow lines of the energy functional on loop space \cite{GGM22}, Floer cylinders in symplectic homology \cite{CGGM23, CGGM24}, and Floer strips with Lagrangian boundary conditions \cite{Meiwes24}. It is central to many results, such as Conley conjecture type results on the multiplicity of periodic points \cite{Batoreo15, GG14, GG16}, dynamics of Hamiltonian and Reeb pseudorotations \cite{GG18, CGGM23}, and the study of topological entropy via barcode invariants \cite{CGG21, GGM22, CGGM24, Meiwes24}.  

Thus, one would like to generalize it for Reeb flows. This was first posed as a question in $2012$ by Ginzburg--G\"{u}rel.  Prior to our work it had not been clear how to prove it for general curves in symplectizations; see e.g. the discussion in \cite[p. 4]{GGM22}. In fact, Theorem~\ref{thm:limit_set_intro}, which uses new tools that did not exist at the time of \cite{GG14}, provides this theorem as a corollary, for any framed Hamiltonian flow and for holomorphic curves with domain any closed Riemann surface with finitely many punctures removed.  Here is the precise statement:
 
 \begin{thm}\label{thm:crossing_energy}
Fix a closed framed Hamiltonian manifold $(Y, \eta)$ and an $\eta$-adapted almost-complex structure $J$. Let $\Lambda$ be a locally maximal $R_\eta$-invariant set, $U$ an isolating neighborhood and $V \subset U$ such that $\overline{V} \subset U$.  Fix an integer $T > 0$ and let $u: C \to \bR \times Y$ be any standard $J$-holomorphic curve with  $\chi(C) \geq -T$.  Then there is a constant $c = c(\eta, J, \Lambda, U, V, T) > 0$ such that
$$\mathcal{A}(u) > c > 0$$
whenever there exists $s_-, s_+ \in \bR$ with
\begin{equation} \label{eq:crossing} u(C) \cap \{s_-\} \times Y \subset V, \quad u(C) \cap \{s_+\}\times Y \not\subset U.\end{equation}
\end{thm}

In analogy with the progress for Hamiltonian diffeomorphisms summarized above, one hopes that Theorem~\ref{thm:crossing_energy} has many potential applications concerning the dynamics of Reeb flows. 

\begin{rem}\label{rem:crossing_energy}
\normalfont
Theorem~\ref{thm:crossing_energy} directly generalizes the crossing energy theorem for Hamiltonian diffeomorphisms from \cite{GG14, GG18}. Given a Hamiltonian diffeomorphism $\phi$, the mapping torus $Y_\phi$ carries a natural framed Hamiltonian structure $\eta$ such that $R_\eta$ generates the suspension flow. There is an explicit correspondence between Floer cylinders for a choice of Hamiltonian $H$ generating $\phi$ and holomorphic cylinders in $\bR \times Y_\phi$\footnote{An explicit derivation for the $2$-disk, which generalizes to arbitrary symplectic manifolds, can be found in \cite[Lemma $20$]{Bramham15}.} The Floer energy of a Floer cylinder is equal to the action of its corresponding holomorphic cylinder. Thus, it suffices to apply Theorem~\ref{thm:crossing_energy} for $Y = Y_\phi$ and pass through this correspondence. On the other hand, there are also very interesting crossing energy theorems proved in the recent works \cite{CGGM23, CGGM24, Meiwes24}. The results in \cite{CGGM23, CGGM24} are for Floer cylinders in completed Liouville domains and the results in \cite{Meiwes24} are for Floer strips with Lagrangian boundary conditions; these do not similarly follow from Theorem~\ref{thm:crossing_energy}.
\end{rem}

\subsection{Further results and remarks}

\subsubsection{The work of Franks and Salazar}

In fact, Theorem~\ref{thm:2d_mid} and Theorem~\ref{thm:3d_mid} follow from slightly more general (but slightly harder to state) results, which we now explain.  This level of generality is also important for the applications to geodesic flows above.

Shortly after the work by Le Calvez--Yoccoz, Franks discovered the following refinement of their theorem. Recall that a compact invariant set $\Lambda$ of a homeomorphism or flow on a compact manifold is called \emph{locally maximal} if any sufficiently $C^0$-close compact invariant set must be contained in $\Lambda$. If a compact invariant set $\Lambda$ is \emph{not} locally maximal, then any neighborhood $U$ of $\Lambda$ contains a point $z \not\in U\,\setminus\,\Lambda$ with orbit closure contained in $U$. Franks \cite{Franks99} showed that for any homeomorphism of $S^2$, the union of periodic points is either infinite or not locally maximal. A subsequent refinement in the conservative case by Salazar \cite{Salazar06} showed that for any area-preserving homeomorphism of $S^2$ and any compact invariant set $\Lambda \subseteq S^2$ containing all periodic points, either $\Lambda = S^2$ or $\Lambda$ is not locally maximal. We are able to generalize these results in the smooth symplectic case as well:

\begin{thm}\label{thm:2d}
Let $\Sigma$ be a closed, oriented surface and let $\phi: \Sigma \to \Sigma$ be any monotone area-preserving diffeomorphism. Then for any compact invariant set $\Lambda \subseteq \Sigma$ containing all periodic orbits of $\phi$, either $\Lambda = \Sigma$ or $\Lambda$ is not locally maximal. 
\end{thm}

\begin{thm}\label{thm:3d}
Let $Y$ be a closed, oriented $3$-manifold equipped with a co-oriented contact structure $\xi$ with torsion first Chern class. Let $\lambda$ be any contact form defining $\xi$ and let $\{\phi^t\}_{t \in \bR}$ denote the Reeb flow of $\lambda$. Then for any compact invariant set $\Lambda \subseteq Y$ containing all periodic orbits of $\{\phi^t\}_{t \in \bR}$, either $\Lambda = Y$ or $\Lambda$ is not locally maximal.
\end{thm}

As we will explain, Theorem~\ref{thm:2d} implies Theorem~\ref{thm:2d_mid}. The analogous chain of reasoning holds starting from Theorem~\ref{thm:3d}.  

\begin{rem}\label{rem:gg}
\normalfont
Related results for Hamiltonian diffeomorphisms of $\mathbb{CP}^n$ with finitely many periodic points and dynamically convex Reeb flows on $S^{2n+1}$ with finitely many closed orbits were respectively proved by Ginzburg--G\"{u}rel \cite{GG18} and Cineli--Ginzburg--G\"{u}rel--Mazzucchelli \cite{CGGM23}.  There is no dimensional restriction in these results, and we say a bit more about this in connection to our results in \S\ref{sec:high}.  
\end{rem}

\subsubsection{The $C^2$-stability conjecture}
\label{sec:cm}

A recent breakthrough result by Contreras--Mazzucchelli \cite{CM21} established the $C^2$-stability conjecture for geodesic flows on Riemannian surfaces, namely that the $C^2$-structurally stable flows are exactly the Anosov ones. A key ingredient in their proof, stated in \cite[Theorem D]{CM21}, was establishing a sufficient criterion for a Kupka--Smale geodesic flow to be Anosov. Our methods give an alternative proof of this result, and also remove the Kupka--Smale assumption:

\begin{cor} \label{cor:anosov_criterion}
Let $F$ be a closed Finsler surface, and let $\mathcal{P} \subseteq SF$ denote the closure of the union of closed orbits of the geodesic flow. If $\mathcal{P}$ is uniformly hyperbolic, then the geodesic flow is Anosov. 
\end{cor}

It is important to note that \cite[Theorem D]{CM21} holds more generally for any Kupka--Smale Reeb flow on a closed contact $3$-manifold. Corollary~\ref{cor:anosov_criterion} can be extended to Reeb flows that are not necessarily Kupka--Smale, but only on closed contact $3$-manifolds with torsion Chern class. The special case of geodesic flows, however, suffices for the application to the stability conjecture. 

The criterion in \cite[Theorem D]{CM21} and Newhouse's classical work \cite{Newhouse77} were also used by Knieper--Schulz \cite{KS22} to prove a different Anosov criterion for geodesic flows; in particular they show Riemannian or Finsler geodesic flows on closed surfaces are Anosov if and only if they are $C^2$-stably ergodic. 

\subsubsection{Non-rational maps}
\label{sec:rational}

We note that, in the case of the torus, our results are close to being sharp. Indeed, any non-rational $\phi$ is either i) Hamiltonian isotopic to a translation $(x, y) \mapsto (x + a, y + b)$ where $(a, b) \not\in \mathbb{Q}^2$ or ii) Hamiltonian isotopic to a smooth conjugate of an affine map 
$$(x, y) \mapsto (x + ny, y + b),$$
 where $n$ is a nonzero integer and $b \not\in \mathbb{Q}$. The examples in case i) have no proper compact invariant sets when both $a$ and $b$ are irrational and the examples in case ii) never have any proper compact invariant subsets.

It would be interesting to see whether our results hold for rational maps in higher genus.  To do this, the main point would be to either remove the assumption of a finite $k$-independent lower bound on $\chi(C_k)$ from Theorem~\ref{thm:limit_set_intro}; or, better understand the topology of ECH/PFH curves, which is of independent interest.

\subsubsection{Three-dimensional energy surfaces}

Some of the phenomena introduced here hold beyond the scope of this paper. A recent preprint \cite{hamiltonians} by the second author proves, among other results, density results such as Corollary~\ref{cor:2d_elementary} and Corollary~\ref{cor:3d_mid} for compact regular Hamiltonian hypersurfaces in $\mathbb{R}^4$. This result was announced in the first version of our preprint. Our Theorem~\ref{thm:limit_set_intro}, or more precisely its main technical precursor Proposition~\ref{prop:local_area_bound}, is one of the key ingredients in \cite{hamiltonians}. 

\subsubsection{Comparison with invariant measures}

The generality of the above theorems strongly precludes the space of possible improvements, without the imposition of additional restrictions. For example, one could hypothetically ask whether the invariant sets we detect support interesting invariant measures.  However, Anosov--Katok famously constructed \cite{AK70} an area-preserving diffeomorphism of $S^2$ whose invariant measures are as simple as possible: the only ergodic invariant measures are a pair of fixed points and the area measure. Corollary~\ref{cor:2d_elementary} applies to the Anosov--Katok example to produce many distinct proper compact invariant sets, but they all must therefore support essentially the same ergodic invariant measures. 

\subsubsection{Higher dimensions}
\label{sec:high}

We close with some speculations on the extension of our results to higher dimensions. As we mentioned in Remark~\ref{rem:gg}, higher-dimensional versions of the theorems in this paper were proved in \cite{GG18, CGGM23}, but with the dynamical assumption that the systems must have finitely many periodic orbits.  

It would be very interesting to find the weakest possible dynamical assumptions for which our theorems extend to higher-dimensional Hamiltonian diffeomorphisms and Reeb flows.  Our Theorem~\ref{thm:limit_set_intro} on the extraction of invariant sets from low-action holomorphic curves with bounded topology works in all dimensions, and as mentioned in Remark \ref{rem:crossing_energy}, implies a ``crossing energy bound'' which is a key technical ingredient (among many) in \cite{GG18}. Proposition~\ref{prop:lim} also clearly works in any dimension. However, our existence results for low-action holomorphic curves rely on deep properties of ECH/PFH, which are invariants defined for area-preserving surface diffeomorphisms and three-dimensional Reeb flows, respectively. 

Finding low-action holomorphic curves with bounded topology in higher dimensions in high generality seems like it will require substantial new ideas.  This is consistent with a more general theme in a range of problems in current symplectic research --- ranging from the kind of questions considered in this paper to problems about the algebraic structure of certain homeomorphism groups to questions like symplectic packing stability \cite{CGHS23, CGH23} --- where one would like analogues of various properties related to ECH/PFH in higher dimensions. 

On a more optimistic note, the fact that Theorem~\ref{thm:limit_set_intro} works for \emph{non-cylindrical} curves is an asset. It opens up for the first time the possibility of using powerful theories such as contact homology or SFT, which count non-cylindrical curves, to explore invariant sets of higher-dimensional Reeb flows.  Indeed, as we have seen here, one needs to consider non-cylindrical curves in our arguments to get our results.

\subsection{Outline of article} \S~\ref{sec:dynamics_proofs} proves Theorems~\ref{thm:2d}, \ref{thm:3d}, and \ref{thm:crossing_energy}. The proofs of Theorems~\ref{thm:2d} and \ref{thm:3d} require Theorem~\ref{thm:limit_set_intro} and two propositions (Propositions~\ref{prop:pfh} and \ref{prop:ech}). These respectively assert that monotone area-preserving maps and Reeb flows of torsion contact forms have many low-action holomorphic curves with controlled topology. \S~\ref{sec:ech} proves Proposition~\ref{prop:ech} using embedded contact homology. \S~\ref{sec:pfh} proves Proposition~\ref{prop:pfh} using periodic Floer homology. \S~\ref{sec:feral_curves} proves Theorem~\ref{thm:limit_set_intro}. 

\subsection{Acknowledgements}  

Crucial discussions around this project occurred at the July $2023$ workshop  \emph{Dynamical Systems} at the Oberwolfach Research Institute.  We thank the institute for their hospitality.  We also thank Erman Cineli, Viktor Ginzburg, Basak G{\"u}rel, and Marco Mazzucchelli for extremely helpful comments on an earlier version of our paper.

DCG: The origins of my interest in the topology of ECH curves are in my joint work \cite{CGHP19} with Michael Hutchings and Dan Pomerleano.  I thank them for many extremely helpful discussions about this topic.   I also thank Umberto Hryniewicz and Hui Liu for further discussions about this around our joint work \cite{CGHHL23}.  I also thank Giovanni Forni for helpful discussions and I think Rich Schwartz for helpful comments on an earlier draft of this work.   In addition, I thank the National Science Foundation for their support under agreements DMS-2227372 and DMS-2238091. 

RP: I am grateful for useful conversations with Joel Fish, Viktor Ginzburg, Helmut Hofer, and Patrice Le Calvez. I would also like to acknowledge support from the National Science Foundation under agreement DGE-1656466 and from the Miller Institute at the University of California Berkeley.

\section{Proofs of main dynamical results}\label{sec:dynamics_proofs}

This section is primarily concerned with the proofs of Theorems~\ref{thm:2d} and \ref{thm:3d}. 
The proofs we give here rely on Theorem~\ref{thm:limit_set_intro} and two propositions (Propositions~\ref{prop:pfh} and \ref{prop:ech}), whose proofs are all deferred to subsequent sections.  After these results are proved, we explain how they imply the other dynamical theorems stated in the introduction. We conclude the section with a proof of Theorem~\ref{thm:crossing_energy}, which is a straightforward corollary of Theorem~\ref{thm:limit_set_intro}.  

Before proceeding, let us explain the basic ideas behind the proofs. Proposition~\ref{prop:ech} asserts that any nondegenerate torsion contact form on a closed $3$-manifold $Y$ admits sequences of holomorphic curves in $\bR \times Y$ with uniformly bounded topology, finite Hofer energy, and arbitrarily low action.  Moreover, the curves in the sequence can be taken to
pass through any point $(0,z)$ in $\bR \times Y$ that is not on any closed Reeb orbit. The low action and bounded topology produce, via Theorem~\ref{thm:limit_set_intro} and Proposition~\ref{prop:lim},  a connected family $\cX$ of compact invariant sets. The finite Hofer energy and point constraints are then exploited, via an elementary topological argument, to conclude Theorem~\ref{thm:3d}. The theorem is proved for degenerate contact forms by approximating them by nondegenerate contact forms and using the resulting holomorphic curves; to make this work, we require a certain amount of quantitative control over the relevant curves, which is why Proposition~\ref{prop:ech} is quantitative in nature. The path from Proposition~\ref{prop:pfh} to Theorem~\ref{thm:2d} is completely analogous. 

\begin{rem}
\normalfont
Our style of argument is robust enough to generalize to other settings where low action holomorphic curves with bounded topology are present. For example, it follows from \cite{GG18} that the mapping torus of any Hamiltonian pseudorotation of $\mathbb{CP}^n$ admits holomorphic cylinders of arbitrarily low action, with finite Hofer energy, passing through any point in the symplectization. Applying our argument proves the analogue of Theorems~\ref{thm:2d} and \ref{thm:3d} for these maps. 
\end{rem}

\subsection{Existence of low-action holomorphic curves with bounded topology}\label{sec:existcurves}

The following two key propositions show that mapping torii of monotone area-preserving surface diffeomorphisms and Reeb flows of torsion $3$-dimensional contact forms have, after possibly making small perturbations, many low-action holomorphic curves with bounded topology. We start with the statement for area-preserving surface diffeomorphisms. Relevant notations and definitions are found in Section~\ref{sec:pfh}. 

\begin{prop}\label{prop:pfh}
Let $\phi$ be a monotone area-preserving diffeomorphism of a closed symplectic surface $(\Sigma, \omega)$. Let $Y_{\phi}$ denote the mapping torus of $\phi$. There exists a positive integer $d_0 \geq 1$, depending only on the Hamiltonian isotopy class of $\phi$, such that the following holds for all $d \geq d_0$ and any nondegenerate Hamiltonian perturbation $\phi'$ of $\phi$: For any fixed $z' \in Y_\phi'$, not on any closed Reeb orbit, and generic choice of $\phi'$-adapted $J'$, there exists a standard $J'$-holomorphic curve $u: C \to \bR \times Y_{\phi'}$ such that:
\begin{enumerate}[(a)]
\item $(0, z') \in u(C)$;
\item $\cE(u) \leq d$;
\item $\cA(u) \leq d^{-1/2}$;
\item $\chi(C) \geq -2$.
\end{enumerate}
\end{prop}

Next, the statement for Reeb flows. Relevant notations and definitions are found in Section~\ref{sec:ech}. 

\begin{prop}\label{prop:ech}
Let $\lambda$ be a torsion contact form on a closed $3$-manifold $Y$. Then there exists a positive integer $k_0 \geq 0$ such that the following holds for any $k \geq k_0$ and any $C^\infty$-small nondegenerate perturbation $\lambda'$ of $\lambda$: For any fixed $z' \in Y$, not on any closed Reeb orbit, and generic choice of $\lambda'$-adapted $J'$, there exists a standard $J'$-holomorphic curve $u: C \to \bR \times Y$ such that:
\begin{enumerate}[(a)]
\item $(0, z') \in u(C)$;
\item $\cE(u) \leq k^{3/4}$;
\item $\cA(u) \leq k^{-1/16}$;
\item $\chi(C) \geq -2$. 
\end{enumerate}
\end{prop}

\begin{rem}
\normalfont
In fact, we will see in the proofs of Proposition~\ref{prop:pfh} and Proposition~\ref{prop:ech} that not only does there exist curves with the above properties, but that the ECH/PFH curves satisfy these properties under the assumptions of the proposition with probability $1$.  More precisely, as we will explain later, for $k$ (resp. $d$) as in the statement of the propositions, the non-triviality of the ``$U$"-map implies the existence of approximately $k$ (resp. $d$) curves, and we can look at the proportion that satisfy the conclusions of the propositions.  Our arguments imply that this number limits to $1$.  
\end{rem}

\subsection{The Hausdorff topology}\label{subsec:hausdorff} Let $Z$ denote any separable and locally compact metric space (e.g. any second countable topological manifold). Recall that $\cK(Z)$ denotes the space of all closed subsets of $Z$, equipped with the topology of Hausdorff convergence. We collect some basic facts about the Hausdorff topology here.  

\subsubsection{Hausdorff convergence} Fix any sequence $\{\Lambda_k\}$ in $\cK(Z)$. We define $\liminf \Lambda_k \in \cK(Z)$ to be the set of $z \in Z$ such that each neighborhood of $z$ intersects all but finitely many of the $\Lambda_k$. We define $\limsup\Lambda_k \in \cK(Z)$ to be the set of $z \in Z$ such that each neighborhood of $z$ intersects infinitely many of the $\Lambda_k$. It is clear that $\liminf\Lambda_k \subseteq \limsup \Lambda_k$. Then $\Lambda_k \to \Lambda$ in the Hausdorff topology if and only if $\liminf \Lambda_k = \Lambda = \limsup \Lambda_k$. We recall \cite[Corollary $2.2$]{McMullen96} that, under the imposed conditions on $Z$, the space $\cK(Z)$ is compact and metrizable.

\subsubsection{Continuity lemmas} The following lemmas are about continuity properties of maps between $\cK(Z)$'s. They are elementary and we omit the proofs. The first lemma asserts that taking a union with a closed set is continuous. 

\begin{lem}\label{lem:union}
Let $Z$ be a separable and locally compact metric space. For any $\Lambda' \in \cK(Z)$, the map $\Lambda \mapsto \Lambda \cup \Lambda'$ is a continuous map from $\cK(Z)$ to itself. 
\end{lem}

%
%
%
%
%

On the other hand, taking an intersection with a closed set is usually not continuous. The next lemma asserts that it is continuous in a rather specific situation. 

\begin{lem}\label{lem:intersection}
Let $Y$ be any separable and locally compact metric space and set $X := (-1, 1) \times Y$, equipped with the product metric. Then for any sequence $\{\Lambda_k\}$ in $\cK(Y)$ such that 
$$(-1, 1) \times \Lambda_k \to (-1, 1) \times \Lambda \in \cK(X),$$ 
we have $\Lambda_k \to \Lambda$. 
\end{lem}

%
%

The last lemma asserts that, when $Z_1$ is a compact metric space, pushing forward a closed set by a continuous map $f: Z_1 \to Z_2$ is Hausdorff continuous. 

\begin{lem}\label{lem:pushforward}
Let $f: Z_1 \to Z_2$ be a continuous map where $Z_1$ is a compact metric space and $Z_2$ is a separable and locally compact metric space. Then the map $\Lambda \mapsto f(\Lambda)$ defines a Hausdorff continuous map $\cK(Z_1) \to \cK(Z_2)$. 
\end{lem}

%
%

\subsubsection{Invariant sets of flows} Suppose that $Z$ is a smooth and compact manifold and let $R$ denote a vector field on $Z$. Let $\cK(Z, R) \subseteq \cK(Z)$ denote the subspace of all closed subsets invariant under the flow of $R$. This subspace is closed and therefore compact, but it may not be connected. The following lemma states and proves an important property of locally maximal $\Lambda \in \cK(Y, R)$. 

\begin{lem}\label{lem:loc_maximal}
If an element $\Lambda \in \cK(Z, R)$ is locally maximal, then it is maximal with respect to inclusion in any connected subspace $\cZ \subseteq \cK(Z, R)$ containing $\Lambda$. 
\end{lem}

\begin{proof}
Conley proved \cite[Theorem $3.5$]{Conley72} that $\Lambda$ is locally maximal if and only if it is maximal with respect to inclusion in some open and closed subspace $\cZ' \subseteq \cK(Z, R)$. Any connected subspace $\cZ$ such that $\Lambda \in \cZ$ must be contained in $\cZ'$ and the lemma follows. 
\end{proof}

\subsection{Proofs of Theorem~\ref{thm:2d} and Theorem~\ref{thm:3d}}

In this section, we prove Theorems~\ref{thm:2d} and Theorem~\ref{thm:3d}.  The proofs are virtually identical, so for brevity we will only give the proof of Theorem~\ref{thm:2d}. 

The main ingredients in the proof of Theorem~\ref{thm:2d} are Theorem~\ref{thm:limit_set_intro}, Proposition~\ref{prop:pfh}, and Proposition~\ref{prop:lim}. Choose any monotone area-preserving diffeomorphism $\phi$ of a closed symplectic surface $(\Sigma, \omega)$. Let $Y_\phi$ denote its mapping torus and $\eta = (dt, \omega_\phi)$ denote its associated framed Hamiltonian structure. We observe that Theorem~\ref{thm:2d} follows from proving its analogue in the mapping torus:

\begin{prop}\label{prop:2d_mapping_torus}
For any closed $R_\eta$-invariant set $\Lambda \subseteq Y_\phi$ containing all periodic orbits of $R_\eta$, either $\Lambda = Y_\phi$ or $\Lambda$ is not locally maximal. 
\end{prop}

\begin{proof}[Proof of Proposition~\ref{prop:2d_mapping_torus}]

Let $\Lambda \subseteq Y_\phi$ be any compact invariant set. We assume without loss of generality that it is proper, since if $\Lambda = Y_\phi$ we are already done.  Our argument will go via approximation to the nondegenerate case, so choose a sequence $\{H_k\}$ of smooth functions $H_k: \bR/\bZ \times \Sigma \to \bR$ converging in $C^\infty$ to $0$ such that for each $k$, the map $\phi_k := \phi \circ \psi^1_{H_k}$ is nondegenerate. For each $k$, the pair $\eta_k := (dt, \omega_\phi + dH_k \wedge dt)$ is a framed Hamiltonian structure, and the map 
$$(t, p) \mapsto (t, (\psi^t_{H_k})^{-1}(p))$$
descends to an isomorphism 
$$(Y_\phi, \eta_k) \to (Y_{\phi_k}, (dt, \omega_{\phi_k}))$$
of framed Hamiltonian manifolds. Since $H_k$ converges to $0$ in the $C^\infty$ topology as $k \to \infty$, we have that $\eta_k$ converges in $C^\infty$ to $\eta$ as $k \to \infty$. For each $k$ we choose an $\eta_k$-adapted $J_k$ such that the sequence $\{(\eta_k, J_k)\}$ converges to a pair $(\eta, J)$ in $\cD(Y)$. 

Now fix any point $z \in Y_\phi$ not in $\Lambda$.  Then there is a sequence of points $\{z_k\}$ converging to $z$, such that $z_k$ is not on any closed Reeb orbit for $\phi_k$; this follows from the fact that the $\phi_k$ are nondegenerate, hence the union of their closed Reeb orbits has measure zero.  By Proposition~\ref{prop:pfh}, after possibly making an arbitrarily small perturbation to each $J_k$, for each sufficiently large $d$ and each sufficiently large $k$ there exists a $J_k$-holomorphic curve
$$u_{d,k}: C_{d,k} \to \bR \times Y_\phi$$
such that:
\begin{enumerate}[(\roman*)]
\item $(0, z_k) \in u_{d,k}(C_{d,k})$;
\item $\cE(u_{d,k}) \leq d$;
\item $\cA(u_{d,k}) \leq d^{-1/2}$;
\item $\chi(C_{d,k}) \geq -2$. 
\end{enumerate}

Write $\cP \subseteq Y_\phi$ for the union of closed orbits of $R_\eta$ and write $\overline{\cP}$ for its closure. For each $d$ and $k$, write $\cP_d(k)$ for the union of closed orbits of $R_{\eta_k}$ of period at most $d$. By the Hofer energy bound in (ii), we have that the level sets concentrate around closed orbits of period $\leq d$ as $s \to \infty$. We also note that since $R_{\eta_k} \to R_\eta$, we have for each fixed $d$ that
$$\limsup_{k \to \infty} \cP_d(k) \subseteq \cP,$$
that is periodic orbits of $R_{\eta_k}$ with bounded period converge to periodic orbits of $R_\eta$. It follows that for each $d$, we can choose some large $k_d \gg 1$ and some $s_d \in \bR$ such that
\begin{equation}
\label{eqn:periodic}
\limsup_{d \to \infty} \tau_{s_d} \cdot \Big(u_{d,k_d}(C_{d,k_d}) \cap (s_d - 1, s_d + 1) \times Y_\phi\Big)\subseteq (-1, 1) \times \overline{\cP}.
\end{equation}

Write $X := (-1, 1) \times Y_\phi$ and write $\cX \subseteq \cK(X)$ for the limit set of the sequence $\{u_{d,k_d}\}_{d \geq 1}$; by passing to a subsequence, we can assume by Proposition~\ref{prop:lim} that this is connected. Let $f: \cK(X) \to \cK(Y_\phi)$ denote the map $\overline{\Lambda} \mapsto \overline{\Lambda} \cap \{0\} \times Y_\phi$ and write $\cY := f(\cX)$. By Lemma~\ref{lem:intersection}, the map $f$ is continuous at each point of $\cX$, so it follows that $\cY$ is a connected subset of $\cK(Y)$. 

By (iii), (iv),  and Theorem~\ref{thm:limit_set_intro}, we have that $\cY \subseteq \cK(Y_\phi, R_\eta)$.  By \eqref{eqn:periodic}, some $\Lambda' \in \cY$ is contained entirely in $\overline{\cP}$, and is therefore contained in $\Lambda$. By (i), some $\Lambda'' \in \cY$ contains the point $z$.  Now let $\cZ \subset \cK(Y_\phi, R_\eta)$ denote the collection of closed invariant sets equal to a union $K \cup \Lambda$ for some $K \in \cX$. By Lemma~\ref{lem:union}, $\cZ$ is the image of a connected set by a continuous map, so it is connected. Moreover, $\Lambda \in \cZ$, since $\Lambda = \Lambda' \cup \Lambda$, and $\Lambda'' \cup \Lambda \in \cZ$ by definition.  However, $\Lambda'' \cup \Lambda$ is not a subset of $\Lambda$, since it contains $z$.  Hence, by Lemma~\ref{lem:loc_maximal}, $\Lambda$ is not locally maximal. 

\end{proof}

\subsection{Proofs of other dynamical results}\label{subsec:other_results}

\subsubsection{Proofs of Theorems~\ref{thm:2d_mid} and \ref{thm:3d_mid}} We prove Theorem~\ref{thm:2d_mid} using Theorem~\ref{thm:2d}. Theorem~\ref{thm:3d_mid}, the version for three-dimensional Reeb flows, follows from the same formal argument using Theorem~\ref{thm:3d}, and so its proof will be skipped for brevity. 

\begin{proof}[Proof of Theorem~\ref{thm:2d_mid}]
Let $\phi: \Sigma \to \Sigma$ be a monotone area-preserving diffeomorphism of a closed surface and let $\Lambda \subset \Sigma$ be a proper closed invariant set. Set $U := \Sigma\,\setminus\,\Lambda$. Let $\cP \subseteq \Sigma$ denote the union of all periodic orbits. If $\cP\not\subseteq \Lambda$, then there exists a periodic orbit contained in $U$. Therefore, $U$ is not minimal. If $\cP \subseteq \Lambda$, then by Theorem~\ref{thm:2d} there exists some proper closed invariant subset $\Lambda'$ arbitrarily $C^0$-close to $\Lambda$ but not equal to $\Lambda$. By taking $\Lambda'$ sufficiently close to $\Lambda$, we can ensure that $\Lambda' \cap U$ is not equal to $U$. Then any point $z \in \Lambda' \cap U$ will not have dense orbit in $U$, so $U$ is not minimal. 
\end{proof}

\subsubsection{Proofs of Corollaries~\ref{cor:2d_elementary} and \ref{cor:3d_mid}} We prove Corollary~\ref{cor:2d_elementary}; the same formal argument proves Corollary~\ref{cor:3d_mid}. 

\begin{proof}[Proof of Corollary~\ref{cor:2d_elementary}]
Let $\phi: \Sigma \to \Sigma$ be a diffeomorphism satisfying the assumptions of Theorem~\ref{thm:2d_mid}. The corollary is equivalent to the statement that there exists infinitely many proper compact $R_\phi$-invariant subsets whose union is dense in the mapping torus $Y_\phi$. Let $\mathcal{K}' := \mathcal{K}(Y_\phi, R_\phi)\,\setminus\,\{Y_\phi\}$ denote the set of all proper compact $R_\phi$-invariant sets. The set $\mathcal{K}'$ admits a natural partial order defined by inclusion. By Theorem~\ref{thm:2d_mid}, $\mathcal{K}'$ has no maximal element. The set $\mathcal{K}'$ is non-empty since $\phi$ has at least one periodic orbit (see Proposition~\ref{prop:ucycle}). The set $\mathcal{K}'$ is infinite, since otherwise it would have a maximal element. Now consider the set
$$Z := \overline{\bigcup_{\Lambda \in \mathcal{K}'}} \Lambda \subseteq Y_\phi,$$ 
the closure of the union of all proper compact invariant sets. The set $Z$ is compact, invariant, and contains all proper compact invariant sets. Since $\mathcal{K}'$ has no maximal element, $Z$ must be equal to $Y_\phi$.
\end{proof}

\subsubsection{Proof of Theorem~\ref{thm:geodesics}} 
\begin{proof}
Fix a closed Finsler surface $F$ as in the statement of the theorem. Now we introduce the notion of a ``projected invariant set'' of the geodesic flow, which will be useful for our proof. Recall that there is a Reeb vector field whose orbits project to geodesics.  We write $SF$ for the unit tangent bundle and $R$ for this vector field. Let $\pi: SF \to F$ denote the bundle projection. We call a compact subset $\Xi \subseteq F$ a \emph{projected invariant set} if there exists some compact $R$-invariant subset $\Lambda \subseteq SF$ such that $\pi(\Lambda) = \Xi$. A projected invariant set $\Xi$ is called \emph{proper} if $\Xi \neq F$. The standard example of a projected invariant set is the closure $\overline{\gamma(\bR)}$ of some geodesic $\gamma$. We use Theorem~\ref{thm:3d} to prove the following analogue of Corollaries~\ref{cor:2d_elementary} and \ref{cor:3d_mid}. 

\begin{lem}\label{lem:geodesics_technical}
There exists infinitely many distinct proper projected invariant sets whose union is dense in $F$. 
\end{lem}

\begin{proof}
Let $\Xi \subseteq F$ be any proper projected invariant set. We claim that there exists a proper projected invariant set $\Xi'$ such that $\Xi$ is a strict subset of $\Xi'$. Once this claim is established, an analogous argument to the proof of Corollary~\ref{cor:2d_elementary} proves the lemma. We omit this part and focus on proving the claim.

Fix a proper projected invariant set $\Xi$ as above. We construct the set $\Xi'$ in two cases, depending on whether $\Xi$ contains all of the closed geodesics or not. First, suppose that there exists a closed geodesic $\gamma: \bR \to F$ such that $\gamma(\bR)$ is not a subset of $\Xi$. Then, we set $\Xi' = \Xi \cup \gamma(\bR)$.  

Next, suppose that $\Xi$ contains the images of all closed geodesics. Let $\Lambda \subset SF$ denote the closure of the union of all compact invariant sets projecting into $\Xi$; note that $\pi(\Lambda) = \Xi$. Since $\Xi$ contains the images of all closed geodesics, it follows from the construction that $\Lambda$ must contain all closed orbits of the geodesic flow. 
By Theorem~\ref{thm:3d} and Lemma~\ref{lem:finsler_torsion}, $\Lambda$ is not locally maximal. Therefore, there exists a Hausdorff convergent sequence $\Lambda_k \to \Lambda$ such that $\Lambda_k$ is not a subset of $\Lambda$ for each $k$. We can assume without loss of generality that $\Lambda$ is a strict subset of $\Lambda_k$ for each $k$. This is done by replacing $\Lambda_k$ with $\Lambda_k \cup \Lambda$ for each $k$; these sets still Hausdorff converge to $\Lambda$ by Lemma~\ref{lem:union}. By the definition of $\Lambda$, it follows that $\Xi$ is a strict subset of $\pi(\Lambda_k)$ for each $k$. Moreover, it follows from Lemma~\ref{lem:pushforward} that $\pi(\Lambda_k) \to \Xi$, so $\pi(\Lambda_k)$ is a proper projected invariant set for sufficiently large $k$. We set $\Xi' = \pi(\Lambda_k)$ for any choice of sufficiently large $k$. 
\end{proof}

We now discuss how to prove Theorem~\ref{thm:geodesics} using the lemma.  By Lemma~\ref{lem:geodesics_technical}, there exists an infinite collection $\mathcal{F} \subseteq \mathcal{K}(F)$ of projected proper invariant sets whose union is dense in $F$. For any proper projected invariant set $\Xi \subset F$ and any point $z \in F$, there exists a geodesic $\gamma: \bR \to F$ such that $\gamma(0) = z$ and $\overline{\gamma(\bR)} \subseteq \Xi$. Hence, for each $\Xi \in \mathcal{F}$, we obtain a countable sequence of geodesics $\gamma^{\Xi}_k$, such that i) $\gamma^\Xi_k(\bR) \subset \Xi$ for each $k$ and ii) $\cup_k \gamma^\Xi_k(\bR)$ is dense in $\Xi$. After passing to a 
subset, we may assume that these geodesics have distinct closures. Let $G$ be the collection of all geodesics $\gamma^\Xi_k$ over all $\Xi \in \mathcal{F}$ and all $k$. Then $G$ is infinite, because $\mathcal{F}$ is, and satisfies Theorem~\ref{thm:geodesics}(a--c). 

\end{proof}

\subsubsection{Proof of Corollary~\ref{cor:anosov_criterion}} 
\begin{proof}
Contreras--Mazzucchelli proved \cite[Section $3.4$, p. $18$]{CM21} that, if the invariant set $\mathcal{P}$ is uniformly hyperbolic, then it is locally maximal. Their argument is short and we paraphrase it here for the convenience of the reader. A uniformly hyperbolic invariant is locally maximal if and only if it has ``local product structure'' \cite[Theorem $6.2.7$]{FH19}.  By the definition of local product structure \cite[Section $3.2$]{CM21}, the set $\mathcal{P}$ has local product structure if and only if for any pair of sufficiently close periodic orbits $\gamma_1$ and $\gamma_2$, the local stable/unstable manifolds intersect in $\mathcal{P}$. Contreras--Mazzucchelli observe that if $\gamma_1$ and $\gamma_2$ are close, then they are in the same homoclinic class. By the Birkhoff--Smale horseshoe theorem \cite[Theorem $6.5.2$]{FH19}, it follows that any intersection of local stable/unstable manifolds is approximated by periodic orbits. 

Now, the corollary follows immediately from our main Theorem~\ref{thm:3d}. Since $\mathcal{P}$ is locally maximal, Theorem~\ref{thm:3d} and Lemma~\ref{lem:finsler_torsion} imply that $\mathcal{P} = SF$. It follows that $SF$ is uniformly hyperbolic, and therefore that the geodesic flow is Anosov. 
\end{proof}

\subsection{Proof of crossing energy theorem}

We prove Theorem~\ref{thm:crossing_energy} using Theorem~\ref{thm:limit_set_intro}, Proposition~\ref{prop:lim}, and the lemmas from this section. 

\begin{proof}[Proof of Theorem~\ref{thm:crossing_energy}]
Assume for the sake of contradiction that the corollary is false. Then there exists a sequence of $J$-holomorphic curves $\{u_k: C_k \to \bR \times Y\}$ satisfying \eqref{eq:crossing}, $\chi(C_k) \geq -T$ and $\mathcal{A}(u_k) \leq 1/k$. By Theorem~\ref{thm:limit_set_intro}, every element of the limit set $\cX$ is a cylinder over a closed invariant set. By Proposition~\ref{prop:lim}, after replacing $\{u_k\}$ with a subsequence if necessary, the limit set is connected. Let $f: \cK(X) \to \cK(Y)$ denote the map $\overline{\Lambda} \mapsto \overline{\Lambda} \cap \{0\} \times Y$ and set $\cY := f(\cX)$. By Lemma~\ref{lem:intersection} and Theorem~\ref{thm:limit_set_intro}, $\cY$ is a connected subspace of $\cK(Y, R_\eta)$.

Let $\cZ \subseteq \cK(Y)$ denote the collection of closed invariant sets which are a union $K \cup \Lambda$ for some $K \in \cY$; this is connected by Lemma~\ref{lem:union}.  The first item in \eqref{eq:crossing} implies $\cY$ contains an invariant set $\Lambda'$ contained in $U$, and since $U$ is an isolating neighborhood of $\Lambda$, this implies $\Lambda' \subseteq \Lambda$. It follows that $\Lambda \in \cZ$. Since $\Lambda$ is locally maximal, every element of $\cZ$ is a subset of $\Lambda$ by Lemma~\ref{lem:loc_maximal}. Hence, every element of $\cY$ is a subset of $\Lambda$. This in turn implies that for sufficiently large $k$, every level set of $u_k(C_k)$ lies inside $\bR \times \overline{V}$. We now arrive at a contradiction by appealing to the second condition of \eqref{eq:crossing}. 
\end{proof}

\section{Low-action curves from embedded contact homology}\label{sec:ech}

This section proves Proposition~\ref{prop:ech}. Previously statements like this have been proved under the assumption of two periodic Reeb orbits \cite{HT09b}, or later under the assumption of finitely many periodic Reeb orbits \cite{CGHP19}; this section shows that this phenomenon in fact holds much more generally.   

The basic strategy of proof follows the strategy developed in \cite{HT09b,CGHP19}, whereby one considers a tower of $U$-curves and compares the ECH index $I$ to the $J_0$ index, which controls the topology of the curves --- we will review these terms below.   What is new here is the argument to bound the difference between $I$ and $J_0$: this is controlled by the Chern class and the Conley-Zehnder index, and some new ideas are needed to bound these terms without assuming finitely many orbits.

\subsection{Embedded contact homology}\label{subsec:ech_overview}

We review the basic features of embedded contact homology \cite{Hutchings02, Hutchings14} here. Fix a closed, smooth, connected, oriented three-manifold $Y$ and a contact structure $\xi$. 

\subsubsection{Reeb flow basics}\label{subsec:reeb_flows} Fix any contact form $\lambda$ defining $\xi$, i.e. satisfying the identity $\ker(\lambda) = \xi$. Recall that the \emph{Reeb vector field} $R$ is the unique vector field solving the equations
$$\lambda(R) \equiv 1,\quad d\lambda(R, -) \equiv 0.$$
A {\em closed Reeb orbit} is a smooth map $\gamma: \bR/T\bZ \to Y$ for some $T > 0$ such that $\dot\gamma(t) = R(\gamma(t))$ for all $t$; as is standard we will make no distinction between two closed Reeb orbits that agree up to a reparameterization of the domain. A closed Reeb orbit is \emph{simple} if $\gamma$ is injective. For any closed Reeb orbit $\gamma$, we write $\gamma^k: \bR/kT\bZ \to Y$ for its $k$-th iteration. The number $T$ is the \emph{action} of $\gamma$, denoted by
$$\mathcal{A}(\gamma) := \int_\gamma \lambda = T.$$

The time $T$ linearized flow of $R$ determines a symplectic isomorphism $\xi_{\gamma(0)} \to \xi_{\gamma(0)}$ called the \emph{Poincar\'e return map}. The orbit $\gamma$ is \emph{nondegenerate} if the return map does not have an eigenvalue equal to $1$. A nondegenerate orbit $\gamma$ is \emph{hyperbolic} if $P_\gamma$ has real eigenvalues and \emph{elliptic} if $P_\gamma$ has complex eigenvalues of unit length. We say $\lambda$ is \emph{nondegenerate} if all closed Reeb orbits are nondegenerate. For any fixed contact structure $\xi$, a generic defining contact form $\lambda$ is nondegenerate. 

\subsubsection{ECH generators} A \emph{Reeb orbit set} is a (possibly empty) finite set $\alpha = \{(\alpha_i, m_i)\}$ of pairs $(\alpha_i, m_i)$, where $\alpha_i$ is a simple closed Reeb orbit and $m_i \in \mathbb{N}$ is a positive integer multiplicity. An \emph{ECH generator} is a Reeb orbit set $\alpha = \{(\alpha_i, m_i)\}$ such that i) each of the $\alpha_i$ are pairwise distinct and ii) $m_i = 1$ if $\alpha_i$ is hyperbolic. Denote by $\ECC(Y, \lambda)$ the $\bZ/2$-vector space generated by the set of ECH generators. Any Reeb orbit set $\alpha$ has a homology class $[\alpha] := \sum_i m_i[\alpha_i] \in H_1(Y; \bZ)$. For each $\Gamma \in H_1(Y; \bZ)$, let $\ECC(Y, \lambda, \Gamma)$ denote the sub-module generated by ECH generators homologous to $\Gamma$. 

\subsubsection{ECH differential} Assume that $\lambda$ is nondegenerate. Choose a generic $\lambda$-adapted almost-complex structure $J$ on $\bR \times Y$. The ECH differential $$\partial_J: \ECC(Y, \lambda) \to \ECC(Y,\lambda)$$ is defined by counting certain ``$J$-holomorphic currents'' which we now define. We say that a $J$-holomorphic curve $u: C \to \bR \times Y$ is \emph{somewhere injective} if there exists $\zeta \in C$ such that $u^{-1}(u(\zeta)) = \{\zeta\}$ and $Du$ is injective at $\zeta$. A \emph{$J$-holomorphic current} is a finite set $\cC = \{(C_k, d_k)\}$ of pairs where the $C_k$ denote distinct standard, somewhere injective $J$-holomorphic curves with finite Hofer energy, and the $d_k$ are positive integer multiplicities. We say that $\cC$ is \emph{somewhere injective} if $d_k = 1$ for each $k$ and \emph{embedded} if the $C_k$ are pairwise disjoint and embedded.  For any $J$-holomorphic current $\cC = \{(C_k, d_k)\}$, the slices 
$$\cC \cap \{s\} \times Y = \{(C_k \cap \{s\} \times Y, d_k)\}$$ 
form for $|s| \gg 1$ a weighted collection of embedded loops in $Y$. The slices converge as $1$-dimensional currents to Reeb orbit sets $\alpha$ and $\beta$ as $s \to \infty$ and $s \to -\infty$, respectively. For any pair of Reeb orbit sets $\alpha$ and $\beta$ with $[\alpha] = [\beta]$, we let $\cM(\alpha, \beta)$ denote the moduli space of $J$-holomorphic curves with positive asymptotic limit at $\alpha$ and negative asymptotic limit at $\beta$. 

Any $\cC \in \cM(\alpha, \beta)$ has an associated \emph{ECH index} $I(\cC) \in \bZ$, defined below, and for each $k \in \bZ$ we let $\cM_k(\alpha, \beta)$ denote the subspace of curves of ECH index $k$. When $J$ is sufficiently generic, the space $\cM_1(\alpha, \beta)$ is a smooth $1$-dimensional manifold. Moreover, it has a free $\bR$-action given by the translation action on $\bR \times Y$, and the quotient $\cM_1(\alpha, \beta)/\bR$ is a finite set of points. The matrix coefficient of the ECH differential with respect to a pair of ECH generators $\alpha$ and $\beta$ is defined by the identity
$$\langle \partial_J\alpha , \beta \rangle := \#_2\cM(\alpha, \beta)/\bR$$
where $\#_2$ denotes the modulo $2$ count of points. By \cite{HT07, HT09a}, $\partial_J^2 = 0$, and therefore $(\ECC(Y, \lambda), \partial_J)$ is a chain complex. The \emph{embedded contact homology} $\ECH(Y, \xi)$ is its homology group. A consequence of Taubes' isomorphism of ECH with monopole Floer homology \cite{Taubes10} is that $\ECH(Y, \xi)$ does not depend on the choice of contact form $\lambda$ defining $\xi$ or the choice of $J$ used to define the ECH differential. The ECH differential preserves the homology class $[\alpha] \in H_1(Y;\bZ)$ of an orbit set $\alpha$; for any $\Gamma \in H_1(Y; \bZ)$ we write $\ECH(Y, \xi, \Gamma)$ for the homology of $(\ECC(Y, \lambda, \Gamma), \partial_J)$. 

\subsubsection{The $U$-map} For a generic choice of $\lambda$-adapted almost-complex structure $J$ and any point $z \in Y$ not on any closed Reeb orbit, there exists a chain map
$$U_{J,z}: \ECC(Y, \lambda,\Gamma) \to \ECC(Y, \lambda,\Gamma)$$
defined as follows. Write $\cM_2(\alpha, \beta; z)$ for the space of all $J$-holomorphic currents with ECH index $2$ whose support contains $(0, z) \in \bR \times Y$. For a generic choice of $J$ and any $z$ not on any closed Reeb orbit, the space $\cM_2(\alpha, \beta; z)$ is a finite set of points. The matrix coefficient of $U_{J,z}$ with respect to $\alpha$ and $\beta$ is defined by the identity
$$\langle U_{J,z}\alpha, \beta \rangle = \#_2\cM(\alpha, \beta; z).$$
The map $U_{J, z}$ descends to a map on homology that we call the \emph{$U$-map}:
$$U: \ECH(Y, \xi,\Gamma) \to \ECH(Y, \xi,\Gamma).$$
The map $U_{J,z}$ may vary with different choices of $J$ and $z$. However, the chain homotopy class of $U_{J,z}$ does not depend on the choice of $J$ and $z$, so the induced map on homology does not depend on the choice of $J$ and $z$. 

\subsubsection{$U$-towers and the volume property}

Fix any $\Gamma \in H_1(Y; \bZ)$. A \emph{$U$-tower} is a sequence of nonzero classes 
$$\{\sigma_k\}_{k \geq 0} \subset \ECH(Y, \xi, \Gamma)$$ such that i) $U\sigma_k = \sigma_{k-1}$ for each $k > 0$ and ii) $U\sigma_0 = 0$. Taubes' isomorphism \cite{Taubes10} and a computation by Kronheimer--Mrowka \cite[Chapter $35$]{KM07} prove that $\ECH(Y, \xi, \Gamma)$ contains a $U$-tower whenever the class $c_1(\xi) + 2\operatorname{PD}(\Gamma) \in H^2(Y; \bZ)$ is torsion. Here $c_1(\xi)$ denotes the first Chern class of $\xi$ with respect to any complex structure which rotates positively with respect to $d\lambda$.  

For any nonzero class $\sigma \in \ECH(Y, \xi)$, define its \emph{spectral invariant} $c_\sigma(\lambda) \in \bR$ to be the infimum of all $L$ such that $\sigma$ is represented by a cycle in $\ECC(Y, \lambda)$ with all constituent generators having action $\leq L$. Here the action of an ECH generator $\alpha = \{(\alpha_i, m_i)\}$ is defined to be
$$\cA(\alpha) := \sum_i m_i\cA(\alpha).$$
A quantitative version of the proof that ECH is independent of the choice of $J$ used to define the ECH differential shows that the spectral invariants do not depend on $J$ either. However, the spectral invariants can and usually do vary with $\lambda$. Each spectral invariant $c_\sigma(\lambda)$ is $C^0$-continuous with respect to $\lambda$; this allows us to extend the definition of $c_\sigma(\lambda)$ to degenerate $\lambda$. The following lemma records some relevant chain-level information that we can extract from a $U$-tower. 

\begin{lem}\label{lem:utower}
Assume that there exists a $U$-tower $\{\sigma_k\}_{k \geq 0} \subset \ECH(Y, \xi, \Gamma)$ for some $\Gamma \in H_1(Y; \bZ)$. Assume that $\lambda$ is nondegenerate and choose generic $J$, and $(0,z)$ not on any closed Reeb orbit, so that the ECH differential $\partial_J$ and the chain map $U_{J,z}$ are well-defined. Then for each $\epsilon > 0$ and each $k \geq 1$, there exists an ECH generator $\alpha_k$ such that
\begin{enumerate}[(a)]
\item $\cA(\alpha_k) \leq c_{\sigma_k}(\lambda)$; 
\item $U_{J,z}^k(\alpha_k) \neq 0$.
\end{enumerate}
\end{lem}

\begin{proof}
For any fixed $\epsilon > 0$ and $k$, there exists a cycle $x \in \ECC(Y, \lambda, \Gamma)$ representing $\sigma_k$ which splits as a sum of ECH generators $x = \sum_{i=1}^N x_i$ with action less than $c_{\sigma_k}(\lambda) + \epsilon$.  
Since there are only finitely many Reeb orbit sets of action $\le c_{\sigma_k}(\lambda) + 1$, as $\lambda$ is nondegenerate, we can therefore find a cycle $x$ as above with action exactly $c_{\sigma_k}(\lambda)$. Since $U_{j,z}^k(x) \neq 0$, it follows that $U_{J,z}^k(x_i) \neq 0$ for some $i$. Set $\alpha_k := x_i$. 
\end{proof}

One of the most powerful properties of the ECH spectral invariants is the \emph{volume property} proved by Cristofaro-Gardiner--Hutchings--Ramos \cite{CGHR15}. Their result, stated in the following proposition, shows that that the spectral invariants of a $U$-tower asymptotically recover the contact volume. 

\begin{thm}[ECH volume property, \cite{CGHR15}] \label{prop:ech_volume}
Assume that there exists a $U$-tower $\{\sigma_k\}_{k \geq 0} \subset \ECH(Y, \xi, \Gamma)$ for some $\Gamma \in H_1(Y; \bZ)$. Then for any contact form $\lambda$ we have
\begin{equation}\label{eq:ech_volume} \lim_{k \to \infty} c_{\sigma_k}(\lambda)^2/2k = \int_Y \lambda \wedge d\lambda.\end{equation}
\end{thm}

\subsubsection{The ECH index} 

We now give the previously deferred definition of the ECH index.  Fix a nondegenerate contact form $\lambda$ and a pair of ECH generators $\alpha = \{(\alpha_i, m_i)\}$ and $\beta = \{(\beta_j, n_j)\}$. Let $H_2(Y, \alpha, \beta)$ denote the space of equivalence classes of integral $2$-chains with boundary $\alpha - \beta$, where two such chains are equivalent if and only if they differ by a $2$-boundary.  The \emph{ECH index} of a class $Z \in H_2(Y, \alpha, \beta)$ is an integer defined by the formula
\begin{equation} \label{eq:ech_index} I(Z) := c_\tau(Z) + Q_\tau(Z) + \sum_i\sum_{k=1}^{m_i}\CZ_\tau(\alpha_i^k) - \sum_j\sum_{l=1}^{n_j}\CZ_\tau(\beta_j^l).\end{equation}
Definitions of these terms can be found in \cite{Hutchings14}. We will narrow our discussion of the ECH index to exactly those terms which are useful for the proof of Proposition~\ref{prop:ech}, namely the relative Chern class and the Conley-Zehnder index. The \emph{relative Chern class} $c_\tau(Z)$ is defined as follows. We choose an oriented smooth surface $S \subset Y$ representing $Z$ and choose a section $\psi: S \to \xi$, transverse to the zero section, such that $\psi$ is a nonzero constant on each component of $\partial S$ and we set 
$$c_\tau(Z) := \#\psi^{-1}(0)$$
where $\#$ denotes the oriented count of points.  As for the Conley-Zehnder index, we define it by
$$\CZ_\tau(\gamma) = \lceil \theta_\tau(\gamma) \rceil + \lfloor \theta_\tau(\gamma) \lfloor$$
where $\theta_\tau(\gamma)$ denotes the ``monodromy number" of the linearized flow along $\gamma$ in the trivialization $\tau$. We refer the reader to \cite{Hutchings14} for a definition of the monodromy number. We define the ECH index of a curve to be the ECH index of its homology class, emphasizing that this is independent of the choice of $\tau$.

\subsubsection{Topological complexity of $U$-map curves}

A variant of the ECH index called the \emph{$J_0$ index} plays a key role in the proof of Propositions~\ref{prop:pfh} and \ref{prop:ech}. In the notation of \eqref{eq:ech_index} we write
\begin{equation} \label{eq:j0_defn} 
\begin{split} J_0(Z) &:= -c_\tau(Z) + Q_\tau(Z) + \sum_i \sum_{k=1}^{m_i-1} \CZ_\tau(\alpha_i^k) - \sum_j \sum_{l=1}^{n_j-1} \CZ_\tau(\beta_j^l) \\
&= I(Z) - 2c_\tau(Z) - (\sum_i \CZ_\tau(\alpha_i^{m_i}) - \sum_j \CZ_\tau(\beta_j^{n_j})). 
\end{split}
\end{equation}

The $J_0$ index controls the topological complexity of a holomorphic current counted by the $U$-map. To state the bound precisely, we need to recall the following structural property of $U$-map currents, a proof of which is found in \cite[Proposition $3.7$]{Hutchings14}. Any current $\cC \in \cM_2(\alpha, \beta; z)$ counted by the $U$-map splits as a disjoint union $\cC_0 \sqcup C_2$ where $\cC_0$ is a union of trivial cylinders with multiplicities and $C_2$ is an embedded $J$-holomorphic curve with $I(C_2) = 2$. 

\begin{prop}[{\cite[Section $6$]{Hutchings09}}] \label{prop:j0_genus}
Fix a generic $J$ and point $z$ so that the chain map $U_{J,z}$ is defined, and let $\cC = \cC_0 \sqcup C_2 \in \cM_2(\alpha, \beta; z)$ be a $J$-holomorphic current counted by the $U$-map.  Then 
\begin{equation} \label{eq:j0_genus} J_0(\cC) \geq -\chi(C_2). \end{equation} 
\end{prop}

\subsection{The based rotation number}

To prove what we need to know about the ECH curves, we will also need to recall some information about the ``rotation number" of flows.  Our treatment here is inspired by, and closely follows \cite{CE22}, though we handle a few points in a different way that is better suited for our purposes. For any closed Reeb orbit $\gamma: \bR/T\bZ \to Y$ and any choice of (positive) symplectic trivialization $\tau$ there is a number $\rho_\tau(\gamma, \xi) \in \bR$ called the \emph{based rotation number} of $\gamma$, which measures how $\xi$ rotates under the linearization of the Reeb flow.  It is defined as follows.  The linearized flow is symplectic; apply polar decomposition and take the unitary part. The unitary part descends to a flow on the oriented real projectivization $P(\xi)$.  Pull back by $\gamma$ and conjugate with the trivialization $\tau$ to define a flow
$$\bar{\Phi}: \bR \times (\bR/T\bZ \times \bR/\bZ) \to \bR/T\bZ \times \bR/\bZ$$
generated by a vector field $\overline{R}$. We write $\theta: \bR/T\bZ \times \bR/\bZ \to \bR/\bZ$ for the angular coordinate on the target. Let $\overline{\theta}: [0, T] \to \bR$ be the unique real-valued lift of the circle-valued map $t \mapsto \theta(\bar{\Phi}_t(0, 0))$ satisfying the initial condition $\overline{\theta}(0) = 0$. Then the based rotation number is 
$$\rho_\tau(\gamma, \xi) := \overline{\theta}(T).$$
This depends only on the homotopy class, relative to endpoints, of the path of symplectic matrices arising from the linearized flow.

The next lemma relates the based rotation number for the contact structure to the Conley--Zehnder index. 
\begin{lem}\label{lem:rho_cz}
For any closed Reeb orbit $\gamma$ and any trivialization $\tau$, we have the bound
\begin{equation} \label{eq:cz_rotation} |\CZ_\tau(\gamma) - 2\rho_\tau(\gamma, \xi)| \leq 6. \end{equation}
\end{lem}

\begin{proof}
The Conley--Zehnder index is defined by
$$\CZ_\tau(\gamma) = \lceil \theta_\tau(\gamma) \rceil + \lfloor \theta_\tau(\gamma) \lfloor$$
where $\theta_\tau(\gamma)$ denotes the monodromy number of $\gamma$ in the trivialization $\tau$. Define $\rho'_\tau(\gamma,\xi)$ analogously to $\rho_\tau(\gamma,\xi)$, but using the full linearized flow rather than just the unitary part.  It is proved in \cite[Lemma $2.6$]{CE22} that
$$|\theta_\tau(\gamma) - \rho'_\tau(\gamma, \xi)| \leq 1.$$
It remains to understand the relationship between $\rho$ and $\rho'$.  We claim that
\begin{equation}
\label{eqn:approx}
|\rho_\tau(\gamma,\xi) - \rho'_\tau(\gamma, \xi)| \leq 1,
\end{equation}
which implies the lemma in view of the above.  To see why \eqref{eqn:approx} holds, we first note that if we choose as our basepoint (i.e. our trivialization $\tau$) an eigenvector of the positive-definite symmetric part of the polar decomposition of the time $T$ linearized flow, then in fact $\rho'_{\tau} = \rho_\tau$.  Indeed, the space of positive-definite symmetric and symplectic matrices is contractible, and the rotation number only depends on the homotopy class, rel endpoints, so we can replace the positive-definite part of the path of matrices arising from the linearized flow by symmetric and symplectic positive-definite matrices which all have $\tau$ as an eigenvector with positive eigenvalue.  The claimed inequality \eqref{eqn:approx} now follows from \cite[Lemma $2.6$]{CE22}, which bounds the difference between the based rotation number measured with respect to two different basepoints. 
\end{proof}

\subsection{Proof of Proposition~\ref{prop:ech}} We can now give the proof of Proposition~\ref{prop:ech}.   

\begin{proof}

\textbf{Step $1$:}  To deal with the fact that we are considering contact structures that are torsion, but possibly non-trivial, we will need to work with an ``$n^{\text{th}}$-power" construction.  This step collects the results we will need about this.

Fix an integer $n \geq 1$ such that $n \cdot c_1(\xi) = 0$. Write $\xi_n = \xi \otimes \ldots \otimes \xi$ for the $n$-fold (complex) tensor product of $\xi$.  This is a (trivial) complex line bundle. Choose a unitary trivialization $\tau$ of $\xi$ over the simple closed Reeb orbits. This induces a trivialization $\tau_n$ of $\xi_n$. The line bundle $\xi_n$ has a relative Chern class, defined analogously to the contact case. We first note that the relative Chern class of $\xi_n$ with respect to $\tau_n$ is computed as follows:
\begin{equation} \label{eq:ctau_additive} c_{\tau_n}(Z, \xi_n) = n \cdot c_\tau(Z, \xi). \end{equation}

Next, it is useful to understand how the Chern class depends on the choice of trivialization. Fix any pair of unitary trivializations $\tau, \tau'$ and any simple closed Reeb orbit $\gamma: \bR/T\bZ \to Y$. The trivializations define unitary bundle isomorphisms 
$$\tau, \tau': \gamma^*\xi_n \to \bR/T\bZ \times \mathbb{C};$$ 
the composition $\tau \circ (\tau')^{-1}$ defines a smooth map $\bR/T\bZ \to U(1)$. Denote the degree of this map by $\operatorname{wind}_\gamma(\tau, \tau'; \xi_n)$. Then we have the identity
\begin{equation} \label{eq:chern_change_triv} c_{\tau}(Z, \xi_n) - c_{\tau'}(Z, \xi_n) = -\sum_i m_i\operatorname{wind}_{\alpha_i}(\tau, \tau'; \xi) + \sum_j n_j\operatorname{wind}_{\beta_j}(\tau, \tau'; \xi_n). \end{equation}
This is proved by the same argument as in the case of contact structures \cite{Hutchings02}.

There is an analogous story for the based rotation number. The unitary part of the linearized flow, being complex linear, defines a map on the complex tensor product $\xi_n$ and we can defined the based rotation number analogously, which we call the {\em induced based rotation number} (in the trivialization $\tau)$ on $\xi_n$, denoted $\rho_{\tau}(\gamma, \xi_n)$. We now prove some basic properties of the induced based rotation number analogous to the observed properties of the relative Chern class. 

\begin{lem}\label{lem:rho_properties}
The unitary component of the based rotation number and the induced based rotation number satisfy the following basic properties:
\begin{itemize}
\item (Change of trivialization) For any pair $\tau$, $\tau'$ of unitary trivializations we have 
\begin{equation} \label{eq:rotation_change_triv} \rho_{\tau}(\gamma, \xi_n) - \rho_{\tau'}(\gamma, \xi_n) = \operatorname{wind}_\gamma(\tau, \tau'; \xi_n). \end{equation}
\item (Additive under tensor product) For any $n \geq 1$ we have 
\begin{equation} \label{eq:rot_additive} \rho_{\tau_n}(\gamma, \xi_n) = n \cdot \rho_\tau(\gamma, \xi). \end{equation}
\end{itemize}
\end{lem}

\begin{proof}
Let $\bar{\Phi}$ and $\bar{\Phi}'$ be the respective flows on $\bR/T\bZ \times \bR/\bZ$ defined by $\tau$ and $\tau'$. After applying a constant rotation to one of the trivializations, we may assume without loss of generality that $\tau^{-1}(0, 0) = (\tau')^{-1}(0, 0)$. It follows that
$$\bar{\Phi}_t(0,0) = (\tau \circ (\tau')^{-1}) \circ \bar{\Phi}'_t(0,0)$$
for each $t \in \bR$. The lifts $\overline{\theta}$ and $\overline{\theta}'$ corresponding to $\tau$ and $\tau'$ differ by the lift of the map $\bR/T\bZ \to U(1)$ defined by $\tau \circ (\tau')^{-1}$. The map $\bR/T\bZ \to U(1)$ has degree $\operatorname{wind}_\gamma(\tau, \tau'; \Xi)$, so it follows that 
$$\theta(T) - \overline{\theta}'(T) = \operatorname{wind}_\gamma(\tau, \tau'; \Xi)$$
which proves \eqref{eq:rotation_change_triv}. 

Now fix any $n \geq 1$ and let $\bar{\Phi}$ and $\bar{\Phi}_n$ be the respective flows on $\bR/T\bZ \times \bR/\bZ$ defined by $\tau$ and $\tau_n$. We observe that 
$$(\bar{\Phi}_n)_t = n \cdot \bar{\Phi}_t$$
for each $t \in \bR$. This implies that $\overline{\theta}_n = n \cdot \overline{\theta}$ where $\overline{\theta}$ and $\overline{\theta}_n$ are the lifts corresponding to $\tau$ and $\tau_n$. Evaluating both sides at $T$ yields \eqref{eq:rot_additive}. 
\end{proof}

\textbf{Step $2$:} Recall that the tensor power $\xi_n$ is a trivial complex line bundle. We fix a \emph{global} unitary trivialization $\ft$ of $\xi_n$ over the entire manifold $Y$. Let $\lambda'$ be any nondegenerate contact form. Let $\alpha = \{(\alpha_i, m_i)\}$ and $\beta = \{(\beta_j, n_j)\}$ be any pair of homologous ECH generators such that $\cA(\beta) \leq \cA(\alpha)$. This step proves that there exists a constant $\delta > 0$ depending only on the background metric, the $C^2$ norm of $\lambda'$, and the trivialization $\ft$ such that for any $Z \in H_2(Y, \alpha, \beta)$ we have
\begin{equation} \label{eq:ech_defect_bound} 
|I(Z) - J_0(Z)| \leq \delta\cA(\alpha).
\end{equation}

Write $K(Z) := I(Z) - J_0(Z)$. Choose any symplectic trivialization $\tau$ of $\xi$ over the simple closed Reeb orbits. Then $K(Z)$ expands as 
$$K(Z) = 2c_{\tau}(Z) + \sum_i \CZ_{\tau}(\alpha_i^{m_i}) - \sum_j \CZ_{\tau}(\beta_j^{m_j}).$$

We define
$$K_{\text{approx}}(Z, \xi) := 2c_{\tau}(Z, \xi) + 2\sum_i \rho_{\tau}(\alpha_i^{m_i}, \xi) - 2\sum_j \rho_{\tau}(\beta_j^{m_j}, \xi)$$
and a corresponding version
\begin{equation}\label{eq:ech_technical_7} K_{\text{approx}}(Z, \xi_n) := 2c_{\tau'}(Z, \xi_n) + 2\sum_i \rho_{\tau'}(\alpha_i^{m_i}, \xi_n) - 2\sum_j \rho_{\tau'}(\beta_j^{m_j}, \xi_n)\end{equation}
for $\xi_n$, where $\tau'$ denotes a unitary trivialization of $\xi_n$ over the simple closed Reeb orbits. It follows from \eqref{eq:chern_change_triv} and \eqref{eq:rotation_change_triv} that the definition of $K_{\text{approx}}(Z, \xi_n)$ does not depend on the choice of $\tau'$. It follows from \eqref{eq:cz_rotation} that
\begin{equation} \label{eq:ech_technical_1} |K(Z) - K_{\text{approx}}(Z, \xi)| \leq 6\sum_i m_i + 6\sum_j m_j \leq 12T_{\text{min}}(\lambda')^{-1}\cA(\alpha)\end{equation}
where $T_{\text{min}}(\lambda')$ denotes the minimal period of a closed Reeb orbit of $\lambda'$. Note that $T_{\text{min}}(\lambda')$ admits a positive lower bound depending only on the $C^2$ norm of $\lambda'$. 

We now bound $K_{\text{approx}}(Z, \xi)$. Since the left-hand side in \eqref{eq:ech_technical_7} does not depend on the choice of $\tau'$ on the right-hand side, we set $\tau' = \tau_n$ and use \eqref{eq:ctau_additive} and \eqref{eq:rot_additive} to show that 
\begin{equation}\label{eq:ech_technical_2} K_{\text{approx}}(Z, \xi) = n^{-1} \cdot K_{\text{approx}}(Z, \xi_n).\end{equation}

Now we set $\tau' = \ft$ and expand
$$K_{\text{approx}}(Z, \xi_n) := 2c_{\ft}(Z, \xi_n) + 2\sum_i \rho_{\ft}(\alpha_i^{m_i}, \xi_n) - 2\sum_j \rho_{\ft}(\beta_j^{m_j}, \xi_n).$$

It follows that immediately that $c_{\ft}(Z, \xi_n) = 0$. It remains to bound $\rho_{\ft}(\gamma, \xi_n)$ for any closed Reeb orbit $\gamma: \bR/T\bZ \to Y$. To do so, it is convenient to observe that the based rotation number along $\gamma$ can be computed by integrating the ``rotation density'' of the flow with respect to $\ft$. To be precise, the global trivialization $\ft$ and the action of the unitary part of the linearized flow on $\xi_n$ define a flow 
$$\bar\Phi_t: \bR \times Y \times \bR/\bZ \to Y \times \bR/\bZ$$
generated by a vector field $\bar{R}$. The Lie derivative of the $\bR/\bZ$-coordinate on the target is a smooth function $r_{\ft}: Y \to \bR$, which restricts to the rotation density on any simple closed Reeb orbit. The function $r_{\ft}$ depends only on $\ft$ and the linearized Reeb flow, so $|r_{\ft}|$ admits a finite upper bound $c > 0$ depending only on $\ft$ and the $C^2$ norm of $\lambda'$. We conclude that
$$|\rho_{\ft}(\gamma, \xi_n)| \leq \sup |r_{\ft}| \cdot \cA(\gamma) \leq \delta_1 \cdot \cA(\gamma).$$

It follows from the above bound that
\begin{equation}\label{eq:ech_technical_4}|K_{\text{approx}}(Z, \xi_n)| \leq 2\delta_1(\sum_i m_i\cA(\alpha_i) + \sum_j n_j\cA(\beta_j)) \leq 4\delta_1\cA(\alpha).\end{equation}

Combine \eqref{eq:ech_technical_1}, \eqref{eq:ech_technical_2}, and \eqref{eq:ech_technical_4} to show
$$|K(Z)| \leq (12T_{\text{min}}(\lambda')^{-1} + 4\delta_1)\cA(\alpha)$$
which proves \eqref{eq:ech_defect_bound} with $\delta := 12T_{\text{min}}(\lambda')^{-1} + 4\delta_1$. 

\vspace{2 mm}

\noindent\textbf{Step $3$:} 
To simplify the notation, write $c'_k = c_{\sigma_k}(\lambda')$. Choose generic $J$ and $z$ not on a closed Reeb orbit such that the chain map $U_{J,z}$ is well-defined. By Lemma~\ref{lem:utower}, for any $k \geq 1$, there exists an ECH generator $\alpha_k$ such that
\begin{enumerate}[(\roman*)]
\item $\cA(\alpha_k) \leq c'_k$;
\item $U_{J,z}^k(\alpha_k) \neq 0$.
\end{enumerate}

It follows that there exists a sequence of ECH generators $\{\beta_j\}_{j=0}^k$, each with $\cA(\beta_j) \leq c'_k$, and $J$-holomorphic currents $\{\cC_j\}_{j=1}^k$ such that $\cC_j \in \cM(\beta_j, \beta_{j-1})$, $I(\cC_j) = 2$, and $(0, z) \in \operatorname{supp}(\cC_j)$ for each $j$. Now set $Z := \sum_{j=1}^k [\cC_j]$. Using the fact that $I(\cC_j) = 2$ for each $j$ and the bound \eqref{eq:ech_defect_bound} we derive the bound

\begin{equation}\label{eq:ech_technical_5} \sum_{j=1}^k J_0(\cC_j) = J_0(Z) \leq 2k + \delta\cA(\alpha_k) \leq 2k + 2\delta c'_k.\end{equation}

It is an immediate consequence of \eqref{eq:j0_genus} that $J_0(\cC_j) \geq -1$ for each $j$. Write $S_1$ for the set of indices $j$ such that $J_0(\cC_j) \geq 3$. It follows that $3S_1 - (k-S_1) \le 2k + \delta c'_k$, hence
\begin{equation}
\label{eqn:S1bound}
\#S_1 \le \frac{3}{4}k + \frac{\delta}{4}c'_k.
\end{equation}

Write $S_2$ for the set of all indices $j$ such that $\cA(\cC_j) \geq k^{-1/16}$. Since $\sum_{j=1}^k \cA(\cC_j) = \cA(\alpha_k) - \cA(\beta_0) \leq c'_k$, and the action is nonnegative, it follows that

\begin{equation}
\label{eqn:S2bound}
\#S_2 \leq c'_k k^{1/8}.
\end{equation} 
The quantity $c_k(\lambda)$ is $O(k^{1/2})$, in view of \eqref{eq:ech_volume}, and for $\lambda'$ sufficiently close to $\lambda$ $c'_k \le 2 c_k(\lambda)$.  

Thus, by \eqref{eqn:S1bound} and \eqref{eqn:S2bound} there exists an index $0 \le j \le k$ in neither $S_1$ nor $S_2$.  Take $C_k$ to be any component of $\cC_j$ passing through $(0, z)$. By \eqref{eq:j0_genus} it follows that $\chi(C_k) \geq -J_0(\cC_j)$. Thus, $C_k$ satisfies the requirements of Proposition~\ref{prop:ech}.  
\end{proof}

\section{Low-action curves from periodic Floer homology}\label{sec:pfh}

This section proves Proposition~\ref{prop:pfh} using the theory of periodic Floer homology (PFH), an analogue of ECH for area-preserving surface maps defined in \cite{Hutchings02, HS05}. The proof is relatively simple compared to the proof of Proposition~\ref{prop:ech} above, which uses deep quantitative properties of ECH. We exploit an algebraic aspect of PFH that is not present in ECH, namely that PFH has many ``$U$-cycles'' \cite{EH21, CGPPZ21, CGP21}. 

\subsection{Periodic Floer homology}\label{subsec:pfh} We review the theory of PFH. We will discuss both the basics and some key new developments in the theory. Fix a closed, oriented surface $\Sigma$ of genus $g$, an area form $\omega$, and a diffeomorphism $\phi: \Sigma \to \Sigma$ preserving the area form.

\subsubsection{Basics} The \emph{mapping torus} of $\phi$ is the $3$-manifold
$$Y_\phi := [0,1] \times \Sigma\,/\,(1, p) \sim (0, \phi(p)).$$

Write $t$ for the coordinate on the interval $[0,1]$. The one-form $dt$ on $[0,1] \times \Sigma$ descends to a closed $1$-form, also denoted by $dt$, on $Y_\phi$. The area form $\omega$ defines a closed two-form $\omega_\phi$ on $Y_\phi$. The pair $\eta = (dt, \omega_\phi)$ is a framed Hamiltonian structure and the Reeb vector field $R_\phi := R_\eta$ generates the suspension flow of $\phi$. The mapping torus $Y_\phi$ fibers over the circle; write $V_\phi \to Y_\phi$ for the vertical tangent bundle. 

Several key definitions carry over to this setting from ECH. In analogy with ECH, we will call periodic orbits of $R_\phi$ \emph{closed Reeb orbits}. The definitions of elliptic/hyperbolic orbits from ECH have analogues here, replacing the bundle $\xi$ with the bundle $V_\phi$. Moreover, the ECH and $J_0$ indices are also defined in this setting, again replacing $\xi$ with $V_\phi$. 

\subsubsection{Rationality and monotonicity}\label{subsec:monotonicity} We say $\phi$ is \emph{rational} if the cohomology class $[\omega_\phi]$ is a real multiple of a rational class. We say that $\phi$ is \emph{monotone} if it is rational and we have $c_1(V_\phi) = c[\omega_\phi]$ for some constant $c \in \bR$. When $g \neq 1$, our monotonicity condition coincides with the monotonicity condition introduced by Seidel \cite{Seidel02}. When $g = 1$, we show in Lemma~\ref{lem:torus_monotone} that $c_1(V_\phi)$ always vanishes, so any rational area-preserving diffeomorphism is monotone in this case. 

\subsubsection{Definition of PFH} The definition of the version of PFH that we will use requires that $\phi$ is rational and also \emph{nondegenerate}, meaning every closed Reeb orbit is either elliptic or hyperbolic. Choose a generic $\eta$-adapted almost-complex structure $J$ on $\bR \times Y_\phi$. Choose a union $\gamma$ of embedded loops, transverse to $Y_\phi$, called a \emph{reference cycle}. Let $\Sigma$ denote the homology class of a fiber of the map $Y_\phi \to S^1$. The \emph{degree} $d(\gamma)$ of $\gamma$ is the oriented intersection number of $\gamma$ with $[\Sigma] \in H_2(Y_\phi; \bZ)$. We assume that $d(\gamma) > \max(0, g - 1)$ and that that $\gamma$ is \emph{monotone}. This means that the homology class $\Gamma := [\gamma] \in H_1(Y_\phi; \bZ)$ satisfies the identity
\begin{equation} \label{eq:monotone} c_1(V_\phi) + 2\operatorname{PD}(\Gamma) = c \cdot [\omega_\phi] \end{equation}
for some constant $c \neq 0$. The constant $c$ is explicitly computable: pairing both sides of \eqref{eq:monotone} with $[\Sigma]$ shows that $c = 2A^{-1}(d - g + 1)$, where $A := \int_\Sigma \omega$. We note that \eqref{eq:monotone} has a solution if and only if $\phi$ is rational, and that if \eqref{eq:monotone} has a solution, it has solutions of arbitrarily high degree. Finally, let $K_\phi := \ker(\omega_\phi)$ denote the subgroup of all integral homology classes on which $\omega_\phi$ integrates to $0$. 

The PFH chain complex $\PFC_*(\phi, \gamma)$ is defined to be the vector space over $\bZ/2$ freely generated by pairs $\Theta = (\alpha, Z)$ that we call \emph{anchored ECH generators}. Here $\alpha = \{(\alpha_i, m_i)\}$ is an ECH generator such that $[\alpha] = [\gamma]$ and $Z$ is an element of $H_2(Y_\phi, \alpha, \gamma) / K_\phi$ (recall that $H_2(Y_\phi, \alpha, \gamma)$ is an affine space over $H_2(Y_\phi; \bZ)$). 

The differential $\partial_J$ is defined similarly to the ECH differential, although now we take the relative homology classes of the holomorphic curves into account. Write $\cM(\alpha, \beta, W)$ for the moduli space of holomorphic currents from $\alpha$ to $\beta$ that represent the class $W \in H_2(Y, \alpha, \beta)$; let $\cM_k(\alpha, \beta, W)$ denote the subspace of currents with ECH index $k$. Fix a pair of anchored ECH generators $\Theta = (\alpha, Z)$, $\Theta' = (\beta, Z')$. The matrix coefficient of $\partial_J$ with respect to $\Theta$ and $\Theta'$ is defined by the formula
$$\langle \partial\Theta, \Theta' \rangle := \#_2\cM_1(\alpha, \beta, Z - Z')/\bR.$$

Write $\PFH_*(\phi, \gamma)$ for the homology of the complex $(\PFC_*(\phi, \gamma), \partial_J)$. The PFH chain complex and homology group carry some additional basic features that we now review. There is a natural action of $H_2(Y_\phi; \bZ)$ on $\PFC_*(\phi, \gamma, J)$; a class $W \in H_2(Y_\phi; \bZ)$ acts on a generator $(\alpha, Z)$ by sending it to $(\alpha, Z + W)$. This action commutes with the differential and so descends to an action on $\PFH_*(\phi, \gamma)$ as well. The $U$-map on PFH is also defined analogously to the $U$-map for ECH. 

After choosing a framing of $V_\phi$ over $\gamma$, the PFH complex also comes equipped with a $\bZ$-grading, which is defined for each anchored ECH generator by the formula
\begin{equation}\label{eq:pfh_grading} I(\Theta) := c_\tau(Z) + Q_\tau(Z) + \sum_i \sum_{k=1}^{m_i} \CZ_\tau(\alpha_i^k). \end{equation}

The differential and $U$-map have degree $-1$ and $-2$ with respect to this grading. The $H_2$-action shifts the grading as follows: 
\begin{equation}\label{eq:h2_grading}
I(W \cdot \Theta) = I(\Theta) + \langle c_1(V_\phi) + 2\operatorname{PD}(\Gamma), W \rangle = I(\Theta) + 2A^{-1}(d - g + 1)\int_W \omega_\phi
\end{equation}
for any anchored ECH generator $\Theta$ and any $W \in H_2(Y_\phi; \bZ)$. The last line uses \eqref{eq:monotone} and our computation of the monotonicity constant above. The identity \eqref{eq:h2_grading} also shows that the $\bZ$-grading is well-defined. 

\subsubsection{The $U$-cycle property} The analogue of a $U$-tower in ECH is a \emph{$U$-cycle}. Assume that $\phi$ is nondegenerate and rational and choose a monotone reference cycle $\gamma$ so that PFH is well-defined. A nonzero element $\sigma \in \PFH_*(\phi, \gamma, G)$ is \emph{$U$-cyclic of order $m$} for some integer $m \geq 1$ if 
$$U^{m(d(\gamma) - g + 1)}\sigma = (-m[\Sigma]) \cdot \sigma.$$

It is known that every nonzero element of PFH is $U$-cyclic as long as $\gamma$ has sufficiently high degree. 

\begin{prop}[Existence of $U$-cyclic elements, \cite{CGPPZ21}] \label{prop:ucycle}
Assume that $\phi$ is nondegenerate and rational and fix a monotone reference cycle $\gamma$. There exists an integer $d_0 > \max(0, g - 1)$, depending only on the Hamiltonian isotopy class of $\phi$, such that if $d(\gamma) \geq d_0,$ then $\PFH_*(\phi, \gamma) \neq 0$ and every nonzero class is $U$-cyclic.
\end{prop}

The following lemma is a ``chain-level'' version of Proposition~\ref{prop:ucycle}. 

\begin{lem}\label{lem:ucycle_chain}
Assume that $\phi$ is nondegenerate and rational and fix a monotone reference cycle $\gamma$. There exists an integer $d_0 > \max(0, g - 1)$, depending only on the Hamiltonian isotopy class of $\phi$, such that the following holds. Choose any monotone reference cycle $\gamma$ such that $d(\gamma) \geq d_0$. Choose generic $J$ and $z \in Y_\phi$ so that the chain-level map $U_{J,z}$ is well-defined. Then there exist positive integers $m_0$ and $m_1$ and a sequence $\{\Theta_j\}_{j=1}^{m_1}$ of nonzero generators of $\PFC_*(\phi, \gamma)$ such that 
$$\langle U_{J,z}^{m_0(d(\gamma) - g + 1)}\Theta_j, \Theta_{j+1} \rangle \neq 0$$
for each $j \in \{1, \ldots, m_1 - 1\}$ and 
$$\langle U_{J,z}^{m_0(d(\gamma) - g + 1)}\Theta_{m_1}, m_0m_1[\Sigma] \cdot \Theta_1 \rangle \neq 0.$$
\end{lem}

\begin{proof}
Suppose $\gamma$ has sufficiently high degree so that Proposition~\ref{prop:ucycle} holds. Choose a trivialization of the restriction of $V_\phi$ to $\gamma$ and use this to define a $\bZ$-grading on PFH. Fix a grading $k$ for which $\PFH_k(\phi, \gamma) \neq 0$. Write $Z_k \subset \PFC_k(\phi, \gamma)$ for the space of cycles of degree $k$, and $B_k$ for the space of boundaries of degree $k$. The proof will take $3$ steps. 

\noindent\textbf{Step $1$:} This step shows that $Z_k$ and $B_k$ have finite dimension over $\bZ/2$. By the change of grading formula \eqref{eq:h2_grading}, it follows that for each ECH generator $\alpha$ there exists at most one anchored ECH generator $\Theta = (\alpha, Z)$ such that $I(\Theta) = k$. Since $\phi$ is nondegenerate, it has finitely many ECH generators representing any given homology class in $H_1(Y_\phi;\mathbb{Z})$. This implies that $\PFC_k(\phi, \gamma)$ contains finitely many anchored ECH generators, so it is a finite-dimensional $\bZ/2$-vector space. This implies that $Z_k$ and $B_k$ have finite dimension as well. 

\noindent\textbf{Step $2$:} This step uses Proposition~\ref{prop:ucycle} to show that there is a nonzero cycle $x \in Z_k$ fixed up to a shift by an iterate of the $U$-map. Fix generic $J$ and $z$ so that the chain-level map $U_{J,z}$ is well defined. Proposition~\ref{prop:ucycle} implies that there exists an integer $m_0 \geq 1$ such that for any nonzero cycle $x \in Z_k$, there exists a chain $z$ such that
\begin{equation}\label{eq:ucycle_chain_1}m_0[\Sigma] \cdot U_{J,z}^{m_0(d(\gamma) - g + 1)}x = x + \partial z.\end{equation}

Let $T$ be the restriction of $m_0[\Sigma] \cdot U_{J,z}^{m_0(d(\gamma) - g + 1)}$ to $Z_k$. Then \eqref{eq:ucycle_chain_1} implies that $\operatorname{Im}(T - 1) \subseteq B_k$. Since $\PFH_k(\phi, \gamma) \neq 0$, it follows that $\operatorname{dim}(Z_k) > \operatorname{dim}(B_k)$. This implies that the operator $T - 1$ has nonzero kernel and therefore there exists some nonzero $x \in Z_k$ such that $Tx = x$. 

\noindent\textbf{Step $3$:} This step completes the proof. Expand the element $x$ from the previous step into a sum $\sum_{i=1}^N x_i$ where each $x_i$ is an anchored ECH generator. The desired cyclic sequence $\{\Theta_j\}_{j = 1}^{m_1}$ of anchored ECH generators will be picked out from the $x_i$ using a short combinatorial argument. Define a directed graph $G$ as follows. The vertex set of $G$ is $\{1, \ldots, N\}$ and there is an edge from $i$ to $j$ if and only if $\langle Tx_i, x_j \rangle \neq 0$. We allow edges to start and end at the same vertex. It is well-known that any directed graph with no \emph{sources}, i.e. vertices which have no incoming edges, has a directed cycle. Now, $G$ has no sources: this follows because $Tx = x$ implies that for each $j$, the identity $\langle Tx, x_j \rangle \neq 0$ holds, which in turn implies that there exists some $i$ such that $G$ has an edge from $i$ to $j$.  Thus, $G$ has a cycle. Thus, there exists a set $\{x'_j\}_{j = 1}^{m_1}$ of anchored ECH generators such that 
\begin{equation}
\label{eqn:teqns}
\langle Tx'_j, x'_{j+1} \rangle \neq 0, \quad \langle Tx'_{m_1}, x'_1 \rangle \neq 0,
\end{equation}
for each $j \in \{1, \ldots, m_1 - 1\}$. For each $j$, set $\Theta_j := -(j-1)m_0[\Sigma] \cdot x'_j$.  Then, the $\Theta_j$ satisfy the conditions of the lemma by \eqref{eqn:teqns}, since $T = m_0[\Sigma] \cdot U_{J,z}^{m_0(d(\gamma) - g + 1)}$ by definition. 
\end{proof}

\subsection{Proof of Proposition~\ref{prop:pfh}} 

We now suppose that $\phi$ is \emph{monotone}, which we recall means $c_1(V_\phi) = c[\omega_\phi]$ for some constant $c \in \bR$. The proof of Proposition~\ref{prop:pfh} is an immediate consequence of the following result, since the monotonicity condition is preserved under Hamiltonian isotopy.

\begin{prop}\label{prop:pfh_technical}
Assume that $\phi$ is nondegenerate and monotone. 
There exists an integer $d_0 \geq \max(0, g - 1)$ depending only on $g$ and the Hamiltonian isotopy class such that for any $z \in Y_\phi$ not on any closed Reeb orbit, and generic $J$, there exists a standard $J$-holomorphic curve $u_d: C_d \to \bR \times Y_\phi$ satisfying the following properties:
\begin{enumerate}[(a)]
\item $(0, z) \in u_d(C_d)$.
\item $\cE(u_d) \leq d$.
\item $\cA(u_d) \leq d^{-1/2}$.
\item $\chi(C_d) \geq -2$. 
\end{enumerate}
\end{prop}

\begin{proof}
Fix $d_0 > 0$ so that Lemma~\ref{lem:ucycle_chain} holds and fix any monotone reference cycle $\gamma$ with degree $d := d(\gamma) \geq d_0$. The proof will take $2$ steps.

\noindent\textbf{Step $1$:}  Fix generic $J$ and $z$ so that the map $U_{J, z}$ is well-defined on $\PFC_*(\phi, \gamma)$.   Let $\{\Theta_j\}_{j=1}^{m_1}$ denote the sequence of generators provided by Lemma~\ref{lem:ucycle_chain}.  Then, by Lemma~\ref{lem:ucycle_chain}, we obtain a sequence $\cC_1, \ldots, \cC_{m_0m_1(d-g+1)}$ of $J$-holomorphic currents counted by the $U$-map such that
\[ \sum_{i=1}^{m_1m_0(d(\gamma) - g + 1)} [\cC_i] =  m_0m_1[\Sigma] \in H_2(Y_\phi,\alpha_1,\alpha_1) = H_2(Y_\phi; \bZ) / K_\phi.\]
By additivity of the action and of $J_0$ we therefore obtain
\begin{equation}\label{eq:pfh_technical2} 
\sum_{i=1}^{m_1m_0(d(\gamma) - g + 1)} \cA(\cC_i) = m_0m_1A, \quad \quad \sum_{i=1}^{m_1m_0(d(\gamma)-g+1)} J_0(\cC_i)  = 2m_0m_1(d(\gamma) + g - 1).
\end{equation}
In the equality for $J_0$, we have used the fact that $\phi$ is monotone, which implies that $c_1(V_\phi)$ has zero pairing with $K_\phi$, together with the fact that $J_0([\Sigma]) = 2(d(\gamma) + g - 1).$

\noindent\textbf{Step $2$:} This step finishes the proof of the proposition. Write $S_1$ for the set of $i$ such that $\cA(\cC_i) > d(\gamma)^{-1/2}$ and write $S_2$ for the set of $i$ such that $J_0(\cC_i) \geq 3$.  Then, by nonnegativity of the action of pseudoholomorphic curves, and \eqref{eq:pfh_technical2}, we have
\begin{equation} \label{eq:pfh_technical4} \#S_1 \leq Am_0m_1(d(\gamma))^{1/2}. \end{equation}

Since the $J_0$ index is bounded below by $-1$, the bound \eqref{eq:pfh_technical2} implies 
\begin{equation} \label{eq:pfh_technical5} \#S_2 \leq m_0m_1(3d(\gamma) +g - 1)/4.
\end{equation}

Thus, after possibly increasing $d_0$, we have the strict inequality
$$\#(S_1 \cup S_2) < m_0m_1(d(\gamma) - g + 1).$$
This implies that there exists some $i$ such that $\cA(\cC_i) \leq (d(\gamma))^{-1/2}$ and $J_0(\cC_i) \leq 2$. Thus, the component $u_d: C_d \to \bR \times Y_\phi$ of $\cC_i$ containing $(0, z)$ has $\cA(u_d) \leq (d(\gamma))^{1/2}$ and $\chi(C_d) \geq -2$, by \eqref{eq:j0_genus}. It remains to show that $\cE(u_d) \leq d(\gamma):$ this follows since, as $dt$ is closed, the integral of $dt$ over any level set of $\cC_i$ is equal to the pairing $\langle dt, [\gamma] \rangle = d(\gamma)$.
\end{proof}

\section{Invariant sets from low-action holomorphic curves}\label{sec:feral_curves}

The purpose of this section is to prove Theorem~\ref{thm:limit_set_intro}. For the remainder of the section, we fix a closed, smooth, connected, oriented manifold $Y$ of odd dimension $2n + 1 \geq 3$. 

\subsection{Notational preliminaries} \label{subsec:feral_prelim} 

Let us begin by reviewing the setup.

\subsubsection{Stable constants} The statements and proofs below will involve several constants which depend on $Y$, $\eta$, and $J$, where $\eta$ is a framed Hamiltonian structure on $Y$ and $J$ is an $\eta$-adapted almost-complex structure. We say that such a constant is \emph{stable} if it can be taken to be invariant under $C^\infty$-small perturbations of $\eta$ and $J$. 

\subsubsection{Geometry of symplectizations} Let $\cD(Y)$ be the space of pairs $(\eta, J)$ where $\eta$ is a framed Hamiltonian structure and $J$ is an $\eta$-adapted almost-complex structure of $\bR \times Y$. We equip $\cD(Y)$ with the topology of $C^\infty$-convergence. That is, a sequence $\{(\eta_k = (\lambda_k, \omega_k), J_k)\}$ in $\cD(Y)$ converges to $(\eta = (\lambda, \omega), J)$ if and only if the sequences $\{\lambda_k\}$, $\{\omega_k\}$, and $\{J_k\}$ converge in the $C^\infty$-topology to $\lambda$, $\omega$, and $J$, respectively. Choose a pair $(\eta = (\lambda, \omega), J) \in \cD(Y)$. To this pair we associate the following translation-invariant and $J$-invariant Riemannian metric on $\bR \times Y$: 
$$g := da \otimes da + \lambda \otimes \lambda + \omega(-, J-).$$
We fix notation for norms of tensors with respect to $g$. For any smooth tensor $\cT$ on $\bR \times Y$, write $|\cT|_g$ for its pointwise $g$-norm, which is a smooth function on $\bR \times Y$. Write $\|\cT\|_g := \sup_{z \in \bR \times Y} |\cT|_g(z)$ for the $C^0$ norm of $\cT$ with respect to $g$.  We fix notation for the metric balls of $g$. Let $\operatorname{dist}_g$ denote the distance function of $g$. Omitting the dependence on $g$ for brevity, we let 
$$\overline{B}_r(z) := \{w \in \bR\times Y\,|\,\operatorname{dist}_g(z, w) \leq r\}$$ 
denote the closed metric ball of radius $r > 0$ centered at $z \in \bR \times Y$. 

\subsubsection{Geometry of $J$-holomorphic curves} Fix a $J$-holomorphic curve $u: C \to \bR \times Y$. We say $u$ is \emph{compact} and \emph{connected} if the domain $C$ is respectively compact and connected. We say $u$ is \emph{generally immersed} if the critical point set $\operatorname{Crit}(u)$ is discrete. This is always true if $C$ is connected and $u$ is not a constant map. We say $u$ is \emph{boundary immersed} if the restriction of $u$ to $\partial C$ is an immersion. We assume for the sake of convenience that any $J$-holomorphic curves is generally immersed and boundary immersed unless stated otherwise.

We let $\gamma := u^*g$ denote the pullback metric on $C$, which is defined at any point $z \in C$ such that $du(z) \neq 0$. In particular, for a generally immersed curve, the metric is defined outside of a discrete subset of points. Let $\alpha := u^*\lambda$ denote the pullback of $\lambda$. Let $|\cT|_\gamma$ and $\|\cT\|_\gamma$ denote the pointwise and $C^0$ norms of a tensor $\cT$ with respect to $\gamma$. For any domain $U \subseteq C$, we define its \emph{area} to be the integral of the volume form of $\gamma$ over the set $U\,\setminus\,\operatorname{Crit}(u)$:
$$\operatorname{Area}_\gamma(U) := \int_{U\,\setminus\,\operatorname{Crit}(u)} \operatorname{dvol}_\gamma.$$

\subsection{The connected-local area bound and its significance} \label{subsec:local_area_statement}

The main estimate required for the proof of Theorem~\ref{thm:limit_set_intro} is a so-called ``connected-local area bound''.  In this section we state this estimate, deferring the proof to Section~\ref{subsec:local_area_proof}, and then use it to prove Theorem~\ref{thm:limit_set_intro}.

Given a $J$-holomorphic curve $u: C \to \bR \times Y$, any point $\zeta \in C$, and any $r > 0$, let $S_r(\zeta)$ denote the connected component of $u^{-1}(\overline{B}_r(u(\zeta)))$ containing $\zeta$.  Our estimate gives an a priori bound on the area of $S_r(\zeta)$ assuming that $\cA(u)$ is small and $r$ is small. The bound depends on the Euler characteristic of $C$, which is the primary reason why Euler characteristic bounds are assumed in Theorem~\ref{thm:limit_set_intro}. 

\begin{prop}[Connected-local area bound for low-action curves]\label{prop:local_area_bound}
Fix $(\eta,J)\in\cD(Y)$. There exists stable constants $\epsilon_7 = \epsilon_7(\eta, J) > 0$ and $\epsilon_8 = \epsilon_8(\eta, J) > 0$ such that the following holds. Let $u:C \to \bR \times Y$ be a standard $J$-holomorphic curve such that $\cA(u) \leq \epsilon_7$. Then for any point $\zeta \in C$, we have the bound 
\begin{equation}\label{eq:local_area_bound}\operatorname{Area}_{\gamma}(S_{\epsilon_8}(\zeta)) \leq \epsilon_8^{-1}(\chi(C)^2 + 1).\end{equation}
\end{prop}

Proposition~\ref{prop:local_area_bound} is inspired by \cite[Thm. 5]{FH23}, the ``asymptotic connected-local area bound" for ``feral" curves, a fundamental result in the new Fish-Hofer theory.  There are several new aspects here.  The main point is that \cite[Thm. 5]{FH23} is an asymptotic result for a fixed curve (of possibly unbounded Hofer energy, but finite action.)  This allows for the reduction to annular curves, whereas in our setting we need to consider much more complicated topologies; new arguments are required for this.  Another novelty is the replacement of the asymptotic condition with a bound on the action instead.  Both here and in \cite{FH23} one also has to be very careful with the constants to ensure stability.

Let us now explain why Theorem~\ref{thm:limit_set_intro} follows from Proposition~\ref{prop:local_area_bound}.  The argument for this appeared in the feral context in \cite[Proposition $4.47$]{FH23}.

\begin{proof}[Proof of Theorem~\ref{thm:limit_set_intro}] 

As the argument from here is essentially the same as \cite[Proposition $4.47$]{FH23} except for minor modifications, we will only provide a sketch highlighting the key points. Fix $(\eta, J) \in \cD(Y)$ and a sequence $\{(\eta_k, J_k)\}$ in $\cD(Y)$ converging to it. Fix a sequence $\{u_k: C_k \to Y\}$ of standard $J_k$-holomorphic curves such that $\lim_{k \to\infty} \cA(u_k) = 0$ and $\inf_{k} \chi(C_k) > -\infty$. Let $\cX \subset \cK(X)$ denote the limit set and choose any $\overline{\Lambda} \in \cX$.   When $k$ is sufficiently large, and therefore $\cA(u_k)$ is sufficiently small, it follows
that
 $(a \circ u_k)(C_k) = \bR$, hence $\overline{\Lambda}$ is non-empty. To see that $\overline{\Lambda} = (-1, 1) \times \Lambda$, where $\Lambda \in \cK(Y)$ is $R_\eta$-invariant, it suffices to show that for any $z := (t, y) \in \overline{\Lambda}$ there exist some $\epsilon > 0$ such that 
$$(t + \tau, y) \in \overline{\Lambda}, \qquad (t, \phi^\tau(y)) \in \overline{\Lambda}$$ 
for any $\tau \in (-\epsilon, \epsilon)$, where $\{\phi^t\}_{t \in \bR}$ denotes the flow of $R_\eta$.  This is proved by a standard application of Fish's target-local Gromov compactness theorem (the version stated in \cite[Theorem $2.36$]{FH23}). Fix points $\zeta_k \in C_k$ such that $u_k(\zeta_k) \to z$ and define local patches $S_k := S_{\epsilon_8}(\zeta_k) \subset C_k$. Proposition~\ref{prop:local_area_bound} gives a $k$-independent upper bound on $\operatorname{Area}_\gamma(S_k)$ in view of the fact that $\cA(u_k) \to 0$, $\{\chi(C_k)\}$ is uniformly bounded, and all constants are stable. The surfaces $S_k$ have uniformly bounded genus as well, because of the uniform bound on $\chi(C_k)$. Thus, we are justified in applying target-local Gromov compactness, and after passing to the limit obtain a non-constant holomorphic curve passing through $z$ and contained in $\overline{\Lambda}$. The curve has action $0$, due to the fact that $\cA(u_k) \to 0$, and hence has tangent plane at any immersed point equal to $\operatorname{Span}(\partial_a, R_\eta)$, from which the desired property follows. 
\end{proof}

\subsection{Properties of the limit set} Before diving into the proof of the connected-local area bound, we provide a proof of Proposition~\ref{prop:lim} from the introduction. We start with a pair of lemmas. To state the first lemma, it is useful to fix the following notation. For any $(\eta', J') \in \cD(Y)$, any standard $J'$-holomorphic curve $u: C \to \bR \times Y$ defines a map
\begin{gather*}
\mathcal{S}_u: \bR \to \cK(X), \\
s \mapsto \tau_s \cdot \Big( u(C) \cap (s - 1, s + 1) \times Y\Big).
\end{gather*}

We make the following assertion:
\begin{lem}\label{lem:jcurve_cts}
For any $(\eta', J') \in \cD(Y)$ and any standard $J'$-holomorphic curve $u: C \to \bR \times Y$, the map $\mathcal{S}_u$ is continuous. 
\end{lem}

The next lemma is a purely topological statement.

\begin{lem}\label{lem:connected_abstract}
    Let $K$ be any compact and metrizable topological space. Let $Z_k \subseteq K$ be a sequence of connected subsets. Then there exists a subsequence $W_j := Z_{k_j}$ such that their set of subsequential limit points is connected.  
\end{lem}

Lemma~\ref{lem:jcurve_cts} has a short and elementary proof.

\begin{proof}[Proof of Lemma~\ref{lem:jcurve_cts}]
It suffices to prove that for any sequence $s_k \to s$ in $\bR$, we have $\cS_u(s_k) \to \cS_u(s)$ in $\cK(X)$. The proof will take $2$ steps. 

\noindent\textbf{Step $1$:} This step proves that $\limsup \cS_u(s_k) \subseteq \cS_u(s)$. Choose any point $z = (t, y) \in \limsup \cS_u(s_k) \subseteq X$. Then any neighborhood of $z$ meets infinitely many of the sets 
$$\cS_u(s_k) = \tau_{s_k} \cdot \Big(u(C) \cap (s_k - 1, s_k + 1) \times Y\Big).$$

After passing to a subsequence, there exists a sequence of points $z_k = (t_k, y_k) \in \cS_u(s_k)$ such that $t_k \to t$ and $y_k \to y$. Choose some small $\epsilon > 0$ such that $t \pm \epsilon \in (-1, 1)$. Note that $\tau_{-s_k}(z_k) \in u(C)$ for each $k$. Since $s_k \to s$ and $t_k \to t$, it follows that 
$$\tau_{-s_k}(z_k) \in u(C) \cap [s + t - \epsilon, s + t + \epsilon] \times Y$$ for sufficiently large $k$. Since $u(C)$ is proper, the set on the right-hand side is closed. The points $\tau_{-s_k}(z_k)$ converge to $\tau_{-s}(z)$, so it follows that
$$\tau_{-s}(z) ) \in u(C) \cap [s+t-\epsilon, s+t+\epsilon] \times Y \subset u(C) \cap (s - 1, s + 1) \times Y$$
and therefore that $z \in \cS_u(s)$. 

\noindent\textbf{Step $2$:} This step proves that $\cS_u(s) \subseteq \liminf \cS_u(s_k)$. Fix any $z = (t, y) \in \cS_u(s)$. We must show that any neighborhood of $z$ has non-empty intersection with all but finitely many of the sets $\cS_u(s_k)$. Write $z_k := \tau_{s_k-s}(z) = (-s_k + s + t, y)$. For all sufficiently large $k$, we have that $-s_k + s + t \in (-1, 1)$, so it follows that $\tau_{-s_k}(z_k) \in u(C) \cap (s_k - 1, s_k + 1) \times Y$. This implies $z_k \in \cS_u(s_k)$ for sufficiently large $k$. Since $z_k \to z$, we conclude that any neighborhood of $z$ meets $\cS_u(s_k)$ for sufficiently large $k$. 

\end{proof}

Next, we prove Lemma~\ref{lem:connected_abstract}.

\begin{proof}[Proof of Lemma~\ref{lem:connected_abstract}]

Let $W_j := Z_{k_j}$ be any subsequence for which there exist points $w_j \in W_j$ converging to some $w \in K$. Let $W$ denote the set of subsequential limit points of $\{W_j\}$. The proof that $W$ is connected will take $4$ steps.

\noindent\textbf{Step $1$:} This step makes some simplifying observations. It suffices to show that, for any pair of disjoint open subsets $U$ and $V$ such that $W \subseteq U \sqcup V$, we have either $W \subseteq U$ or $W \subseteq V$. Without loss of generality, we will assume that the point $w \in W$ is contained in $U$, and then show that $W \subseteq U$. 

\noindent\textbf{Step $2$:} This step proves that $W_j \subseteq U \sqcup V$ for all sufficiently large $j$. We prove it by contradiction. Assume that the claim is false. Then, after passing to a subsequence, there exists for each $j$ some $w_j' \in W_j$ such that $w_j' \not\in U \sqcup V$ for any $j$. Since $K$ is compact and metrizable, it is sequentially compact, and therefore a subsequence of $\{w_j'\}$ converges to some $w' \in W$. However, the complement of $U \sqcup V$ is closed, so $w'$ does not lie in $U \sqcup V$, contradicting the initial assumption that $W \subseteq U \sqcup V$. 

\noindent\textbf{Step $3$:} This step proves that $W_j \subseteq U$ for all sufficiently large $j$.  By Step $2$ and the fact that $W_j$ is connected for each $j$, we either have $W_j \subseteq U$ or $W_j \subseteq V$ for large $j$. Since $w_j \to w \in U$, it follows that $W_j$ has non-empty intersection with $U$ for sufficiently large $j$, which implies $W_j \subseteq U$. 

\noindent\textbf{Step $4$:} This step completes the proof. By Step $3$, it follows that $W$ lies inside the closure $\overline{U}$. Any open set intersecting $\overline{U}$ must intersect $U$, so $V$ is disjoint from $\overline{U}$. We conclude that $W \subseteq U$. 

\end{proof}

We now give the promised proof of Proposition~\ref{prop:lim}. 

\begin{proof}[Proof of Proposition~\ref{prop:lim}]
Fix $(\eta, J) \in \cD(Y)$ and a sequence $\{(\eta_k, J_k)\}$ in $\cD(Y)$ converging to it. Fix a sequence $\{u_k: C_k \to \bR \times Y\}$ of standard $J_k$-holomorphic curves. Recall that each curve $u_k$ defines a continuous map $\cS_k := \cS_{u_k}$ from $\bR$ to $\cK(X)$. Write $Z_k := \cS_k(\bR)$ for each $k$. By Lemma~\ref{lem:jcurve_cts}, $Z_k$ is the image by a continuous map of a connected space, so it is connected. Thus, $\{Z_k\}$ is a sequence of connected subsets of the compact and metrizable space $\cK(X)$. The limit set $\cX$ is equal to the set of subsequential limit points of the sequence $\{Z_k\}$. The proposition now follows from Lemma~\ref{lem:connected_abstract}.
\end{proof}

\subsection{Preliminaries from feral curve theory}

It remains to prove Proposition~\ref{prop:local_area_bound}, which will take up the remainder of the paper.  To do this, we need to first collect and review some more preliminaries from the work of Fish--Hofer \cite{FH23}, which is the purpose of this section.

\subsubsection{Perturbed holomorphic curves}

Many of the estimates in \cite{FH23} are easiest when the height function $a \circ u$ of the pseudoholomorphic curve is Morse.  Unfortunately, this is not in general the case, so the notion of a ``perturbed curve" was introduced there to remedy this.  More precisely, a \emph{perturbed $J$-holomorphic curve} is a pair $(u, f)$ where $u: C \to \bR \times Y$ is a $J$-holomorphic curve and $f: C \to \bR$ is a smooth function which is compactly supported in the open subset $C\,\setminus\,(\partial C \cup \operatorname{Crit}(u))$.  Perturbing $u$ in the vertical direction by $f$ defines a new map 
$$\widetilde{u}: \zeta \mapsto \operatorname{exp}^g_{u(\zeta)}(f(\zeta)\partial_a).$$
Write $\widetilde{\gamma} := \widetilde{u}^*g$ for the induced pullback metric. Define an almost-complex structure $\widetilde{j}$ on $C$ as the unique almost-complex structure which is a $\widetilde{\gamma}$-isometry and coincides with $j$ on the complement of $\operatorname{supp}(f)$. We then define a one-form 
$$\widetilde{\alpha} := -(\widetilde{u}^*da \circ \widetilde{j})$$
on $C$. This should be thought of as a perturbation of $\alpha = u^*\lambda = -(u^*da \circ j)$. 

As in \cite{FH23}, the perturbation is subject to some quantitative controls.  This is encoded in the notion of a \emph{$(\delta, \epsilon)$-tame perturbation}, which is found in \cite[Definition $4.24$]{FH23}. We will not repeat the precise definition of a $(\delta, \epsilon)$-tame perturbation, since we never make explicit use of it, but we will recall its key attributes. The constant $\delta > 0$ controls the support of $f$ and the domain where $a \circ \widetilde{u}$ is Morse. That is, $f$ is supported outside a $\delta/2$-neighborhood of $\operatorname{Crit}(u)$, and $a \circ \widetilde{u}$ is Morse outside of a $\delta$-neighborhood. The constant $\epsilon$ controls the $C^2$ norm of $f$. We require $\delta$ to be smaller than a constant depending on $u$ and ambient geometry, and $\epsilon$ to be smaller than a constant depending on $\delta$, $u$, and ambient geometry. 

Given a tame perturbation $(u, f)$, several explicit estimates are proved in \cite[Section $4.1$]{FH23} relating the perturbed map $\widetilde{u}$ to the unperturbed map $u$. We do not make direct use of the majority of them, although they are key ingredients in the proofs of other estimates from \cite{FH23} that we cite. One estimate which we do use relates the perturbed and unperturbed pullback metrics:

\begin{lem}[{\cite[Equation $4.3$]{FH23}}]\label{lem:metric_comparison}
Fix $(\eta, J) \in \cD(Y)$. Fix a $J$-holomorphic curve $u: C \to \bR \times Y$ and any $(\delta,\epsilon)$-tame perturbation $(u, f)$. Then we have
$$\gamma/2 \leq \widetilde{\gamma} \leq 2\gamma,$$
where we recall that $\gamma = u^*g$ and $\widetilde{\gamma} = \widetilde{u}^*g$. 
\end{lem}  

We also remark that if the curve $u$ is immersed, then we can use much simpler perturbations, because there are no singularities for the perturbation to avoid. Call a perturbation \emph{$\epsilon$-Morse} if $a \circ \widetilde{u}$ is Morse over the entire domain $C$ and if the $C^2$ norm of $f$ is smaller than $\epsilon$. Our main dynamical results (Theorems~\ref{thm:2d_mid}--\ref{thm:geodesics}) only require working with embedded holomorphic curves. A reader who is solely interested in these results can replace all instances of $(\delta, \epsilon)$-tame perturbations with $\epsilon$-Morse perturbations, where $\epsilon$ is very small.  Lemma~\ref{lem:metric_comparison} is also straightforward if $u$ is immersed. In this case $\widetilde{\gamma}$ is $\epsilon$-close to $\gamma$ in the $C^1$ topology. In the non-immersed setting, some care is required because $\gamma$ degenerates as it approaches $\operatorname{Crit}(u)$.

\subsubsection{Tracts} Fix a perturbed $J$-holomorphic curve $(u, f)$. Tracts and strips, introduced in \cite{FH23}, are highly structured compact portions of the domain $C$. They are important because they satisfy several geometric estimates. A \emph{tract} in $(u, f)$ is a connected, compact embedded surface $\widetilde{C} \subset C$, possibly with boundary and corners, such that:
\begin{enumerate}[(\roman*)]
\item The boundary $\partial\widetilde{C}$ is disjoint from $\operatorname{Crit}(a \circ \widetilde{u})$; 
\item The boundary of $\widetilde{C}$ decomposes as $\partial_h \widetilde{C} \cup \partial_v \widetilde{C}$ where
\begin{enumerate}[(a)]
\item $\partial_h\widetilde{C} \cap \partial_v\widetilde{C}$ is a finite set of corners;
\item If it is non-empty, each component of $\partial_v\widetilde{C}$ is tangent to the gradient vector field $\operatorname{grad}_{\widetilde{\gamma}}(a \circ \widetilde{u})$;
\item The function $a \circ \widetilde{u}$ is constant on each component of $\partial_h\widetilde{C}$.
\end{enumerate}
\end{enumerate}

The sets $\partial_h\widetilde{C}$ and $\partial_v\widetilde{C}$ are respectively called the \emph{horizontal} and \emph{vertical} boundaries of the tract $\widetilde{C}$. 

\subsubsection{Strips}\label{subsec:strips} For the proof of Proposition~\ref{prop:tract_decomposition} at the end of this section, we need to introduce the notion of a strip. Strips, which are closely related to tracts, were originally defined in \cite[Definition $4.17$]{FH23}; we write down a simpler version of their definition which is sufficient for our purposes. A \emph{strip} (in our sense) in $(u, f)$ is an embedded rectangle $\widetilde{C} \subset C$, i.e. a compact domain of genus zero with four smooth boundary curves meeting in four corners, satisfying the following two properties. First, $\widetilde{C}$ contains no critical points of $a \circ \widetilde{u}$. Second, the boundary $\partial\widetilde{C}$ decomposes as a union $\partial_h\widetilde{C} \cup \partial_v\widetilde{C}$ with certain properties. The set $\partial_v\widetilde{C}$ is called the \emph{vertical boundary}, and consists of a pair of disjoint smooth curves tangent to $\operatorname{grad}(a \circ \widetilde{u})$. The set $\partial_h\widetilde{C}$ is called the \emph{horizontal boundary}, and consists of a pair of disjoint smooth curves, which each meet the segments of $\partial_v\widetilde{C}$ at the corners. An illustration is provided in Figure~\ref{fig:tracts_and_strips}. 
We label the two components of $\partial_h\widetilde{C}$ with the notation $\partial_h^{\pm}\widetilde{C}$. The labelling is uniquely determined by the condition $\sup_{\zeta \in \partial_h^+\widetilde{C}} (a \circ \widetilde{u})(\zeta) > \sup_{\zeta \in \partial_h^-\widetilde{C}} (a \circ \widetilde{u})(\zeta)$. We call $\partial_h^+\widetilde{C}$ the \emph{top horizontal boundary} and $\partial_h^-\widetilde{C}$ the \emph{bottom horizontal boundary}. 

The strips we use below are always defined by fixing a segment $\mathcal{I}$ mapping into a level set of $a \circ \widetilde{u}$ and taking a union of gradient flow segments starting at points $\zeta \in \mathcal{I}$. The resulting strip $\widetilde{C}$ will have $\partial_h^-\widetilde{C} = \mathcal{I}$ and $\partial_h^+\widetilde{C}$ consisting of the endpoints of these gradient flow segments. 

\begin{figure}
\includegraphics[width=.4\textwidth]{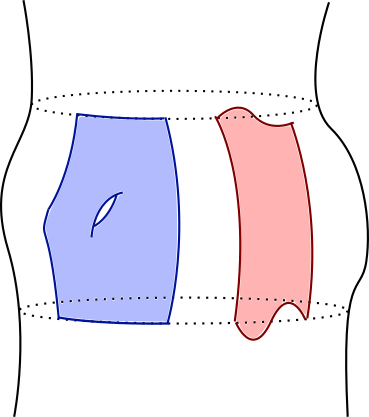}
\caption{A tract (blue) and a strip (red).}
\label{fig:tracts_and_strips}
\end{figure}

\subsubsection{Exponential area bound for tracts} In the classical theory of $J$-holomorphic curves in symplectizations, area bounds are deduced from Hofer energy bounds. The following result provides an alternative in the absence of Hofer energy bounds. It is arguably the most crucial estimate in \cite{FH23}. For example, it shows up in the proofs of Proposition~\ref{prop:fh_quantization}, Lemma~\ref{lem:strip_estimate}, and Lemma~\ref{lem:modest_flow_lines} below. 

\begin{prop}[{\cite[Theorem $9$]{FH23}}] \label{prop:fh_area_bound}
Fix $(\eta, J) \in \cD(Y)$. Let $(u, f)$ be a $(\delta, \epsilon)$-tame perturbed $J$-holomorphic curve with domain $C$ and let $\widetilde{C} \subset C$ be a tract of $(u, f)$, for which there exist constants $a_+ > a_-$ such that
\begin{enumerate}[(\roman*)]
\item $(a \circ \widetilde{u})(\widetilde{C}) \subset [a_-, a_+]$;
\item $(a \circ \widetilde{u})(\partial_h\widetilde{C}) \cap (a_-, a_+) = \emptyset$; 
\item $a_+$ and $a_-$ are regular values of $a \circ \widetilde{u}$. 
\end{enumerate}

Then there exists a stable constant $c_1 = c_1(\eta, J) \geq 1$ such that 
\begin{equation} \label{eq:fh_area_bound2} \operatorname{Area}_{\widetilde{\gamma}}(\widetilde{C}) = \leq c_1 e^{c_1(a_+ - a_-)}\Big(\int_{(a \circ u)^{-1}(a_-) \cap \widetilde{C}} \widetilde{\alpha} + \cA(u)\Big).\end{equation}
\end{prop}

Given its importance, we briefly summarize the proof of Proposition~\ref{prop:fh_area_bound} for the reader. One can compute 
$$\operatorname{Area}_{\widetilde{\gamma}}(\widetilde{C}) = \int_{\widetilde{C}} \widetilde{u}^*da \wedge \widetilde{\alpha} + \widetilde{u}^*\omega = \int_{a_-}^{a_+} \Big(\int_{(a \circ \widetilde{u})^{-1} \cap \widetilde{C}} \widetilde{\alpha}\Big) dt + \cA(u).$$

Write $e(t)$ for the $\widetilde{\alpha}$-measure of $(a \circ \widetilde{u})^{-1}(t) \cap \widetilde{C}$. Whenever the interval $[t_-, t_+]$ contains no critical values, one proves the bound $e(t_+) - e(t_-) \lesssim \int_{t_-}^{t_+} e(s) ds + \delta$, where $\delta$ is some controllable error term. Using Gronwall's inequality, it follows that $e(t)$ grows exponentially in $t$. This is plugged into the computation above to deduce \eqref{eq:fh_area_bound2}. 

\subsubsection{Strip estimates} We now state two technical estimates for strips of perturbed $J$-holomorphic curves, which are used only in the proof of Proposition~\ref{prop:tract_decomposition} at the end of this section. Lemma~\ref{lem:strip_estimate}, which is a simpler version of a lemma from \cite{FH23}, shows that the length of the bottom horizontal boundary of a strip is controlled by the length of the top horizontal boundary, provided that the strip is not too tall. 

\begin{lem}[{\cite[Lemma $4.21$]{FH23}}] \label{lem:strip_estimate}
Fix $(\eta, J) \in \cD(Y)$. Then there exists a stable constant $c_2 = c_2(\eta, J) \geq 1$ such that the following holds. Let $\widetilde{C}$ be a strip of a $(\delta,\epsilon)$-tame perturbed $J$-holomorphic curve $(u, f)$ such that $a \circ \widetilde{u}$ is constant on $\partial_h^-\widetilde{C}$ and 
$$\sup_{\zeta \in \widetilde{C}} (a \circ \widetilde{u})(\zeta) - \inf_{\zeta \in \widetilde{C}} (a \circ \widetilde{u})(\zeta) \leq c_2^{-1}$$

Then we have the bound
$$\int_{\partial_h^{-}\widetilde{C}} \widetilde{\alpha} \leq 2\int_{\partial_h^{+}\widetilde{C}} \widetilde{\alpha} + 2c_2\int_{\widetilde{C}} u^*\omega.$$
\end{lem}

Lemma~\ref{lem:strip_estimate} is proved using Proposition~\ref{prop:fh_area_bound}. To state the next lemma, we need another definition. A strip $\widetilde{C}$ is \emph{rectangular} if both $\partial_h^{\pm}\widetilde{C}$ are contained inside level sets of $a \circ \widetilde{u}$. The lemma asserts that for any finite collection of strips which are not too tall or too short, one of them contains a gradient flow line of $a \circ \widetilde{u}$ of controlled length running from the bottom horizontal boundary component to the top horizontal boundary component. 

\begin{lem}[{\cite[Lemma $4.23$]{FH23}}] \label{lem:modest_flow_lines}
Fix $(\eta, J)\in\cD(Y)$. There exists a stable constant $c_3 = c_3(\eta, J) \geq 1$ such that the following holds. Let $(u, f)$ be a $(\delta,\epsilon)$-tame perturbed $J$-holomorphic curve.  Let $\{\widetilde{C}_k\}_{k=1}^n$ be a finite set of rectangular strips of $(u, f)$ satisfying the following properties:
\begin{enumerate}[(\roman*)]
\item $a_0 = \inf_{\zeta \in \widetilde{C}_k} (a \circ \widetilde{u})(\zeta)$ is independent of $k$;
\item $a_1 = \sup_{\zeta \in \widetilde{C}_k} (a \circ \widetilde{u})(\zeta)$ is independent of $k$;
\item $a_1 - a_0 \leq c_3^{-1}$;
\item $\sum_{k=1}^n \int_{\widetilde{C}_k} u^*\omega \leq (a_1 - a_0)\sum_{k=1}^n \int_{\partial_h^-\widetilde{C}_k} \widetilde{\alpha}$;
\item Each of the strips $\{\widetilde{C}_k\}_{k = 1}^n$ are pairwise disjoint. 
\end{enumerate}

Then there exists $k \in \{1, \ldots, n\}$ and a smooth map $q: [0, s] \to \widetilde{C}_k$ such that
$$\dot q(s) = \operatorname{grad}_{\widetilde{\gamma}}(a \circ \widetilde{u})(q(s)), \quad (a \circ \widetilde{u} \circ q)(0) = a_0, \quad (a \circ \widetilde{u} \circ q)(s) = a_1$$
and 
$$\operatorname{length}_{\widetilde{\gamma}}(q([0,s])) \leq c_3(a_1 - a_0).$$
\end{lem}

The proof of Lemma~\ref{lem:modest_flow_lines} is rather intricate, since the lengths of gradient flow lines could be heavily distorted in the areas where the norm of $\operatorname{grad}(a \circ \widetilde{u})$ is small. We refer the reader to \cite[Section $4.3.1$]{FH23} for details. 

\subsubsection{Action quantization} 
Fish--Hofer \cite[Theorem $4$]{FH23} proved that a holomorphic curve $u:C \to \bR \times Y$ has a positive lower bound on its action near any interior global maximum/minimum of the function $a \circ u$. Their lower bound depends on the genus of $C$, which suffices for our intended applications. For the sake of a cleaner statement, we note that the bound can be made genus-independent using the compactness theory of $J$-holomorphic currents \cite[Remark $5.20$]{Prasad23b}. We now state the precise quantization result. 

\begin{prop}[{\cite[Theorem $4$]{FH23}}]  \label{prop:fh_quantization} 
Fix $(\eta, J)\in\cD(Y)$. Fix any real number $s > 0$. There exists a stable constant 
$\hbar = \hbar(\eta, J, s) > 0$ such that, for any compact, connected $J$-holomorphic curve $u: C \to \mathbb{R} \times Y$, we have 
$$\mathcal{A}(u) \geq \hbar > 0$$ 
provided that the following properties are satisfied for some $a_0 \in \bR$:
\begin{enumerate}[(\roman*)]
\item Either $\min_{\zeta \in C}(a \circ u)(\zeta)$ or $\max_{\zeta \in C}(a \circ u)(\zeta)$ is equal to $a_0$;
\item $(a \circ u)(\partial C) \cap [a_0 - s, a_0 + s] = \emptyset$.
\end{enumerate}
\end{prop}

Proposition~\ref{prop:fh_quantization} is proved via a contradiction argument using Proposition~\ref{prop:fh_area_bound} and Fish's target-local Gromov compactness theorem \cite{Fish11}. 

\subsubsection{Geodesic distance lemma} The following elementary lemma is used in the proof of Lemma~\ref{lem:tract_geometry_2} below. 

\begin{lem}[{\cite[Lemma $4.29$]{FH23}}] \label{lem:geodesic_distance}
Fix $(\eta,J)\in\cD(Y)$. There exists a stable constant $\epsilon_4 = \epsilon_4(\eta, J) \in (0, 1/100)$ such that the following holds for any $\epsilon \leq \epsilon_4$. Each smooth unit-speed immersion $q: [0,T] \to \bR \times Y$ such that
\begin{enumerate}[(\roman*)]
\item $\lambda(\dot q(t)) > 0$ for each $t \in [0,T]$;
\item $\epsilon \leq \int_q \lambda \leq 10\epsilon$;
\item the set of $t \in [0,T]$ such that $\lambda(\dot q(t)) < 1/2$ has Lebesgue measure at most $\epsilon$
\end{enumerate}
satisfies the bound
$$\operatorname{dist}_g(q(0),q(T)) > \epsilon/2.$$
\end{lem}

The bound in Lemma~\ref{lem:geodesic_distance} is proved by direct computation in geodesic local coordinates $y^1, \ldots, y^{2n+1}$, with the additional condition that $\lambda$ is close to $dy^1$. We need to state Lemma~\ref{lem:geodesic_distance} here, instead of alongside Lemma~\ref{lem:tract_geometry_2}, because the constant $\epsilon_4$ appears in the statement of Proposition~\ref{prop:local_area_bound_technical} below. 

\subsection{Proof of the connected local area bound}
\label{subsec:local_area_proof}

We now prove Proposition~\ref{prop:local_area_bound}.  In fact, this will be deduced from a more technical bound, Proposition~\ref{prop:local_area_bound_technical}, which we now state. The statement requires defining a stable constant 
\begin{equation}\label{eq:local_area_const} \epsilon_5 := 2^{-24}\min(c_2^{-1}, c_3^{-4}, \epsilon_4) \end{equation}
where $c_2$, $c_3$, $\epsilon_4$ are the stable constants from Lemmas~\ref{lem:strip_estimate},~\ref{lem:modest_flow_lines}, and ~\ref{lem:geodesic_distance}, respectively.

\begin{prop}\label{prop:local_area_bound_technical}
Fix $(\eta, J) \in \cD(Y)$. Let $\epsilon_4$ and $\epsilon_5$ denote the stable constants from Lemma~\ref{lem:geodesic_distance} and \eqref{eq:local_area_const}, respectively. There exists a stable constant $c_6 = c_6(\eta, J) \geq 1$ such that the following holds. Let $u: C \to \mathbb{R} \times Y$ be a compact, connected $J$-holomorphic curve satisfying the following properties:
\begin{enumerate}[label=(L\arabic*)]
\item \label{label:l1} $(a \circ u)(\partial C) = \{a_0, a_1\}$ where $a_1 > a_0$;
\item \label{label:l2} $a_0$ and $a_1$ are regular values of the projection $a \circ u: C \to \mathbb{R}$; 
\item \label{label:l3} $a_1 - a_0 \geq \epsilon_5/8$;
\item \label{label:l4} $\sup_{\zeta \in C} (a \circ u)(\zeta) - \inf_{\zeta \in C} (a \circ u)(\zeta) \leq \epsilon_5$;
\item \label{label:l5} $\cA(u) \leq 2^{-48}\epsilon_4\epsilon_5$; 
\item \label{label:l6} The set of all $\zeta \in (a \circ u)^{-1}(a_0)$ such that $|u^*\lambda|_\gamma(\zeta) < 1/2$ has Lebesgue measure at most $\epsilon_4$ in $(a \circ u)^{-1}(a_0)$, where Lebesgue measure is defined using the pullback metric $\gamma$.
\end{enumerate}

Then for each $\zeta \in C$, we have the bound
$$\operatorname{Area}_\gamma(S_{\epsilon_5}(\zeta)) \leq c_6(\chi(C)^2 + 1).$$
\end{prop}

Proposition~\ref{prop:local_area_bound_technical} generalizes an estimate proved by Fish--Hofer (stated in \cite[Proposition $4.30$]{FH23}). They made the additional assumption that the domain $C$ is homeomorphic to a compact annulus. The main novelty in the proof of Proposition~\ref{prop:local_area_bound_technical} is the introduction of combinatorial and topological arguments to deal with non-annular holomorphic curves. Let us defer the proof for the moment and first prove Proposition~\ref{prop:local_area_bound} assuming that Proposition~\ref{prop:local_area_bound_technical} holds. The idea is that, assuming the action of $u$ is sufficiently small, any point $\zeta \in C$ is contained inside a surface satisfying the assumptions of Proposition~\ref{prop:local_area_bound_technical}. In particular, to ensure that \ref{label:l6} holds, we need the following technical lemma, which asserts that most tangent planes of a low-action holomorphic curve are nearly vertical. 

\begin{lem}\label{lem:tangent_nearly_vertical}
Write $\mathcal{Q} \subset \bR$ for the set of $t$ such that i) $t$ is a regular value of $a \circ u$ and ii) the set of all $\zeta \in (a \circ u)^{-1}(t)$ such that $|\alpha|_\gamma(\zeta) < 1/2$ has Lebesgue measure greater than $\epsilon_4$. Then the Lebesgue measure of $\mathcal{Q}$ is at most $2\epsilon_4^{-1}\cA(u)$. 
\end{lem}

\begin{proof}
This follows immediately from \cite[Lemma $4.27$]{FH23} with parameters $\delta = \epsilon_4$ and $\theta = 1/2$. 
\end{proof}

We now give the proof of Proposition~\ref{prop:local_area_bound}. 

\begin{proof}[Proof of Proposition~\ref{prop:local_area_bound}]

Let $u: C \to \bR \times Y$ be a standard $J$-holomorphic curve. Pick any point $\zeta \in C$. We will show that, when $\cA(u)$ is sufficiently small, there exists a compact surface $\widetilde{C} \subset C$ containing $\zeta$ in the middle which satisfies \ref{label:l1}--\ref{label:l6} and has Euler characteristic bounded below by $\chi(C)$. Proposition~\ref{prop:local_area_bound_technical} then implies the area bound. The proof will take $6$ steps. 

\noindent\textbf{Step $1$:} This step defines the surface $\widehat{C}$. To define the surface, we need to assume that $\cA(u) \leq 2^{-48}\epsilon_4\epsilon_5$. Now we choose a positive parameter $r \in (\epsilon_5/8, \epsilon_5/4)$ and set $a_0 := (a \circ u)(\zeta) - r$ and $a_1 := (a \circ u)(\zeta) + r$. We fix $r$ so that i) both $a_0$ and $a_1$ are regular values of $a \circ u$ and ii) the set of all $\zeta' \in (a \circ u)^{-1}(a_0)$ such that $|u^*\lambda|_\gamma(\zeta') <  1/2$ has Lebesgue measure at most $\epsilon_4$. There exists a full measure set of $r$ satisfying i) by Sard's theorem, and a positive measure set of $r$ satisfying ii) by Lemma~\ref{lem:tangent_nearly_vertical} and our assumed bound on $\cA(u)$, so an $r$ satisfying both conditions exists. 

Set $C_0 := (a \circ u)^{-1}([a_0, a_1]) \subset C$. Write $\Delta$ for the disjoint union of compact components of $C\,\setminus\,\operatorname{Int}(C_0)$ and set $C_1 := C_0 \cup \Delta$. Let $\widehat{C}$ be the connected component of $C_1$ containing $\zeta$. 

\noindent\textbf{Step $2$:} This step proves the two-sided bound 
\begin{equation}\label{eq:local_area_bound2}
2 \geq \chi(\widehat{C}) \geq \chi(C).
\end{equation}

The upper bound follows from the fact that $\widehat{C}$ is connected. To prove the lower bound, we must show $\chi(\Sigma) \leq 0$ where $\Sigma := C \,\setminus\,\operatorname{Int}(\widehat{C})$. Since $C$ is connected and has empty boundary, it follows from the inverse function theorem that any proper embedding of a surface with empty boundary into $C$ is surjective. It follows that each connected component of $\Sigma$ must have at least one boundary component. Since each connected component of $\Sigma$ intersects $\widehat{C}$, it cannot entirely consist of components of $C_0$ and $\Delta$. It follows that each connected component of $\Sigma$ is non-compact. We have shown that each connected component of $\Sigma$ is non-compact and has non-empty boundary, so $\Sigma$ has non-positive Euler characteristic. 

\noindent\textbf{Step $3$:} Write $\widehat{u}$ for the restriction of $u$ to $\widehat{C}$. This step verifies that $\widehat{u}$ satisfies \ref{label:l2} and \ref{label:l3}. Condition \ref{label:l2} is satisfied becaause $a_0$ and $a_1$ are regular values of $a \circ u$. Condition \ref{label:l3} is satisfied because $r \in (\epsilon_5/8, \epsilon_5/4)$. 

\noindent\textbf{Step $4$:} Assume that $\cA(u) < \hbar(\eta, J, \epsilon_5/8)$ where $\hbar$ denotes the constant from Proposition~\ref{prop:fh_quantization}. This step verifies that $\widehat{u}$ satisfies \ref{label:l1} and \ref{label:l4} given this assumption. 

Write
$$L_+ := \sup_{z \in \widehat{C}}(a \circ u)(z) - \sup_{z \in \partial\widehat{C}}(a \circ u)(z),\quad L_- := \inf_{z \in \partial\widehat{C}}(a \circ u)(z) - \inf_{z \in \widehat{C}}(a \circ u)(z).$$

Note that both $L_+$ and $L_-$ are positive by definition. By the contrapositive of Proposition~\ref{prop:fh_quantization}, it follows that both $L_+$ and $L_-$ are at most $\epsilon/8$. Since $(a \circ \widehat{u})(\zeta) = (a_1 + a_0)/2$, we can rearrange the bounds for $L_+$ and $L_-$ to conclude that
$$\sup_{z \in \partial\widehat{C}} (a \circ \widehat{u})(z) \geq (a_1 + a_0)/2 - \epsilon_5/8 > a_0,\quad \inf_{z \in \partial\widehat{C}} (a \circ \widehat{u})(z) \leq (a_1 + a_0)/2 + \epsilon_5/8 < a_1.$$

Since $(a \circ \widehat{u})(\widehat{C}) \subseteq \{a_0, a_1\}$ by construction, the bounds above are sufficient to conclude \ref{label:l1}. To prove \ref{label:l4}, we use \ref{label:l1} and the upper bounds for $L_+$ and $L_-$:
\begin{equation*}
\sup_{z \in \widehat{C}}(a \circ u)(z) - \inf_{z \in \widehat{C}}(a \circ u)(z) = L_+ + L_- + (a_1 - a_0) \leq \epsilon_5/8 + \epsilon_5/8 + \epsilon_5/2 = \epsilon_5.
\end{equation*}

\noindent\textbf{Step $5$:} This step verifies that $\widehat{u}$ satisfies \ref{label:l5} and \ref{label:l6}. Both follow from our assumptions and the construction in Step $1$. 

\noindent\textbf{Step $6$:} By Proposition~\ref{prop:local_area_bound_technical}, we have the bound \begin{equation} \label{eq:local_area_bound1} \operatorname{Area}_\gamma(S_{\epsilon_5/16}(\zeta) \cap \widehat{C}) = \operatorname{Area}_\gamma(S_{\epsilon_5/16}(\zeta)) \leq c_0(\chi(\widehat{C})^2 + 1) \end{equation} 
where $c_0 = c_0(\eta, J) > 0$ is a stable constant. The area bound then follows from \eqref{eq:local_area_bound2}. 

\end{proof}

It remains to prove Proposition~\ref{prop:local_area_bound_technical} and this will take up the remainder of the paper.  As the proof gets rather technical, let us start by providing a sketch  of the argument.

\vspace{2 mm}

{\em Outline of the proof of Proposition~\ref{prop:local_area_bound_technical}}.

Let $u: C \to \bR \times Y$ denote a compact, connected $J$-holomorphic curve satisfying \ref{label:l1}--\ref{label:l6}. We begin by constructing a \emph{tract decomposition} of $u$ (Proposition~\ref{prop:tract_decomposition}, see also Figure~\ref{fig:tract_decomposition}). Recall that a \emph{tract} is a compact embedded surface with corners in $C$, with \emph{horizontal} boundary components mapping into level sets and \emph{vertical} boundary components mapping into gradient flow lines of the function $a \circ u$. We show that, after perturbing $u$ slightly, the domain $C$ can be cut up into tracts with each boundary component having controlled length. The existence of such a decomposition, for annular domains, is itself implicit in the first five steps of the proof of \cite[Proposition $4.30$]{FH23}, though the explicit statement we formulate here is novel and care is required to get the right statement. Our proof mostly follows these steps, with a new argument to take care of the fact our domain might not be annular; an expository emphasis is also to isolate the key constants to clarify that they are stable.  We defer it to the end of the body of the paper.  

The next step after construction of the tract decomposition is to bound the area of each tract. We show that the number of vertical and horizontal boundary components are each controlled by $\chi(C)$ (Lemma~\ref{lem:tract_topology}) as is the total Euler characteristic of all the tracts (Lemma~\ref{lem:tract_topology_2}). Recalling that each horizontal boundary component has controlled length, we conclude that each tract has a uniform bound on the length of its entire bottom boundary, depending only on $\chi(C)$ and ambient geometry. Applying Proposition~\ref{prop:fh_area_bound} bounds the area of each tract by a constant depending only on $\chi(C)$ and ambient geometry. 

The final step is to cover $S_{\epsilon_5}(\zeta)$ by a controlled number of tracts. This gives a bound on its area since we have already bounded the area of each tract. A geodesic distance argument that we learned from \cite{FH23} implies that $S_{\epsilon_5}(\zeta)$ cannot intersect both vertical boundary components of a ``rectangular" tract, defined (analogously to rectangular strips) below. The topological lemmas mentioned above imply that most tracts are rectangular. These results are combined with a graph-theoretic argument to prove the desired covering bound.

\subsubsection{Statement of tract decomposition}

We now begin the process of making the above outline rigorous. The first part is the statement of the tract decomposition.

\begin{prop} \label{prop:tract_decomposition}
Fix $(\eta, J) \in \cD(Y)$. Let $u: C \to \mathbb{R} \times Y$ be a compact, connected $J$-holomorphic curve satisfying \ref{label:l1}--\ref{label:l6}. There exists a fixed constant $\delta_0 = \delta_0(u, \eta, J) > 0$ such that the following holds. For any sufficiently small $\epsilon > 0$, there exists a $(\delta_0, \epsilon)$-tame perturbation $(u, f)$ and a finite set of tracts $\{\widetilde{C}_k\}_{k=1}^N$ satisfying the following properties:
\begin{enumerate}[(a)]
\item $C = \bigcup_{k=1}^N \widetilde{C}_k$.
\item For each $k \neq k'$, the intersection $\widetilde{C}_k \cap \widetilde{C}_{k'}$ is either empty or equal to a disjoint union of components of $\partial_v\widetilde{C}_k$. 
\item For each $k$, we have $(a \circ \widetilde{u})(\partial_h\widetilde{C}_k) = \{a_0, a_1\}$.
\item For each $k$ and each component $L \in \pi_0(\partial_h\widetilde{C}_k)$ such that $(a \circ \widetilde{u})(L) = a_0$, we have $\int_L \widetilde{\alpha} < 10\epsilon_4$. Moreover, if $L$ is not a circle, then $\int_L \widetilde{\alpha} > \epsilon_4$. 
\item For each $k$ and each component $L \in \pi_0(\partial_v\widetilde{C}_k)$, we have
$$\operatorname{length}_{\widetilde{\gamma}}(L) \leq c_3(a_1 - a_0).$$
\end{enumerate}
\end{prop}

The simplest kinds of tracts in the decomposition are \emph{rectangular} tracts, i.e. tracts with zero genus, two horizontal boundary components, and two vertical components. The tract decomposition could, however, contain tracts with positive genus and many horizontal/vertical boundary components. See Figure~\ref{fig:tract_decomposition} for a schematic of what the tract decomposition might look like. We defer the proof of Proposition~\ref{prop:tract_decomposition} for the moment, collecting some useful lemmas about the asserted tract decomposition first.

\begin{figure}
\includegraphics[width=.6\textwidth]{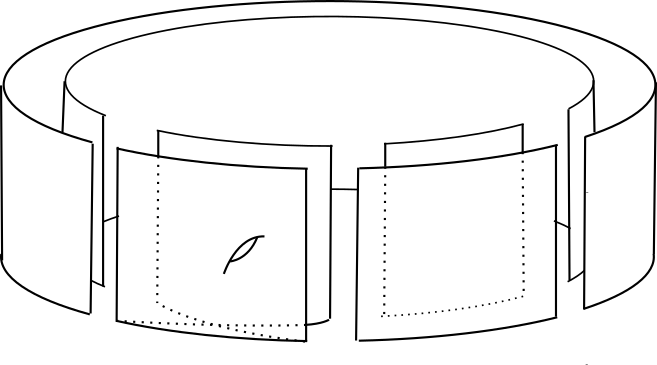}
\caption{A schematic of the tract decomposition. One of the displayed tracts has positive genus and several horizontal/vertical boundary components.}
\label{fig:tract_decomposition}
\end{figure}

\subsubsection{Tract topology bounds} Fix $(\eta, J) \in \cD(Y)$ and a compact, connected $J$-holomorphic curve satisfying \ref{label:l1}--\ref{label:l6}. Use Proposition~\ref{prop:tract_decomposition} to construct a perturbation $(u, f)$ and tract decomposition $\{\widetilde{C}_k\}_{k=1}^N$. As mentioned above, the tracts $\widetilde{C}_k$ could have complicated topology. The next two lemmas provide some a priori topological control. Lemma~\ref{lem:tract_topology} bounds the number of horizontal and vertical boundary components of each tract in terms of $\chi(C)$. 

\begin{lem}\label{lem:tract_topology}
For each $k$, we have the bounds
\begin{equation}\label{eq:tract_topology}\#\pi_0(\partial_v \widetilde{C}_k) \leq 2 - 2\chi(C),\quad \#\pi_0(\partial_h\widetilde{C}_k) \leq 4 - 3\chi(C).\end{equation}
\end{lem}

\begin{proof}
Fix any $k$. Write $M := \#\pi_0(\partial_v \widetilde{C}_k)$ for the number of vertical boundary components of $\widetilde{C}_k$. Let $\widehat{C}$ denote the closure of $C \,\setminus\,\widetilde{C}_k$. The proof will take $2$ steps. 

\noindent\textbf{Step $1$:} This step proves the first bound in \eqref{eq:tract_topology}. We assume without loss of generality that $M > 2$. If $M \leq 2$, then the desired bound is immediate because by \ref{label:l1} we have $\chi(C) \leq 0$. It follows from the Mayer--Vietoris sequence that 
$$M = \chi(\widetilde{C}_k) + \chi(\widehat{C}) - \chi(C).$$

We show $\chi(\widetilde{C}_k) \leq 1$ and $\chi(\widehat{C}) \leq M/2$. The upper bound on $\chi(\widetilde{C}_k)$ follows from the fact that $\chi(\widetilde{C}_k)$ is connected and has non-empty boundary. The upper bound on $\chi(\widehat{C})$ is deduced as follows. Each connected component of $\widehat{C}$ is a tract which shares at least at least two vertical boundary components with $\widetilde{C}_k$. Therefore, $\widehat{C}$ has at most $M/2$ connected components. Each connected component of $\widehat{C}$ has non-empty boundary and therefore has Euler characteristic $\leq 1$. Combine both of these upper bounds with the identity for $M$ and re-arrange to get the first bound in \eqref{eq:tract_topology}. 

\noindent\textbf{Step $2$:} This step proves the second bound in \eqref{eq:tract_topology}. Each connected component of $\partial_h\widetilde{C}_k$ is either i) a compact interval which intersects exactly two components of $\partial_v \widetilde{C}_k$ or ii) a circle which is a connected component of $\partial C$. Any component of $\partial_v\widetilde{C}_k$ intersects exactly two components of $\partial_h\widetilde{C}_k$, so there are $M$ components of the former type, which we showed in Step $1$ is bounded above by $2 - 2\chi(C)$. There are at most $\#\pi_0(\partial C)$ components of the latter type. This gives the bound 
$$\#\pi_0(\partial_h\widetilde{C}_k) \leq 2 + \#\pi_0(\partial C) - 2\chi(C).$$

The second bound in \eqref{eq:tract_topology} now follows from plugging in the inequality $\#\pi_0(\partial C) \leq 2 - \chi(C)$. 
\end{proof}

The next lemma collects an elementary identity, and some useful related observations, used in the proof of Lemma~\ref{lem:tract_geometry} below. 

\begin{lem}\label{lem:tract_topology_2}
The Euler characteristic of the domain $C$ satisfies the following identity:
\begin{equation} \label{eq:tract_topology_2} 2\chi(C) = \sum_{k=1}^N (2\chi(\widetilde{C}_k) - \#\pi_0(\partial_v\widetilde{C}_k)).\end{equation}

Moreover, for each $k$ we have $2\chi(\widetilde{C}_k) \leq \#\pi_0(\partial_v\widetilde{C}_k)$, with equality if and only if either 
\begin{enumerate}[(a)]
\item $\chi(\widetilde{C}_k) = 1$ and $\#\pi_0(\partial_v\widetilde{C}_k) = 2$;
\item $C = \widetilde{C}_k$ and $\chi(\widetilde{C}_k) = \#\pi_0(\partial_v\widetilde{C}_k) = 0$. 
\end{enumerate}
\end{lem}

\begin{proof}
The proof of the lemma will take $3$ steps. 

\noindent\textbf{Step $1$:} This step proves \eqref{eq:tract_topology_2}. For each $k$, write $M_k := \#\pi_0(\partial_v\widetilde{C}_k)$ for the number of vertical boundary components of $\widetilde{C}_k$. Write $M$ for the number of gradient trajectories of $a \circ \widetilde{u}$ which are vertical boundary components of some $\widetilde{C}_k$. It follows that $2M = \sum_k M_k$ because each gradient trajectory is a vertical boundary component of exactly two tracts. The Mayer--Vietoris sequence then implies
\begin{equation*} 2\chi(C) = \sum_k 2\chi(C_k) - 2M = \sum_k (2\chi(C_k) - M_k).\end{equation*}

\noindent\textbf{Step $2$:} Fix any $k$. The next two steps prove that $2\chi(\widetilde{C}_k) \leq \#\pi_0(\partial_v\widetilde{C}_k)$ and characterizes the equality cases. This step provides a proof assuming that $\#\pi_0(\partial_v\widetilde{C}_k) = 0$, i.e. the tract has no vertical boundary components. By Proposition~\ref{prop:tract_decomposition}(b) it follows that $\widetilde{C}_k$ is both closed and open in $C$. Since $C$ is connected, it follows that $C = \widetilde{C}_k$. From Proposition~\ref{prop:tract_decomposition}(c), we conclude that that $\#\pi_0(\partial\widetilde{C}_k) \geq 2$. This implies the bound 
$$2\chi(\widetilde{C}_k) \leq 0 = \#\pi_0(\partial_v\widetilde{C}_k),$$
with equality if and only if $\chi(\widetilde{C}_k) = \#\pi_0(\partial_v\widetilde{C}_k) = 0$. 

\noindent\textbf{Step $3$:} This step has the same aim as the previous step, but in the case where $\#\pi_0(\partial_v\widetilde{C}_k) > 0$. In this case, we must have $\#\pi_0(\partial_v\widetilde{C}_k) \geq 2$, and that $\widetilde{C}_k$ has non-empty boundary, so we conclude that
$$2\chi(\widetilde{C}_k) \leq 2 \leq \#\pi_0(\partial_v\widetilde{C}_k).$$

Equality holds if and only if $\chi(\widetilde{C}_k) = 1$ and $\#\pi_0(\partial_v\widetilde{C}_k) = 2$. 
\end{proof}

\begin{rem}\label{rem:tract_topology}
\normalfont
A tract $\widetilde{C}$ is \emph{rectangular} if it has zero genus, two horizontal boundary components, and two vertical boundary components. A tract is rectangular if and only if $\chi(\widetilde{C}) = 1$ and $\#\pi_0(\partial_v\widetilde{C}) = 2$. The primary implication of the identity \eqref{eq:tract_topology_2} is that, barring a degenerate case, the tract $\widetilde{C}_k$ is rectangular for all but at most $-\chi(C)$ indices $k$. 
\end{rem}

\subsubsection{Tract coverings of controlled size} As above, fix a perturbed $J$-holomorphic curve $(u, f)$ and a tract decomposition $\{\widetilde{C}_k\}_{k=1}^N$. The following lemma asserts that local connected components can be covered by a controlled number of tracts. 

\begin{lem}\label{lem:tract_geometry}
For any point $\zeta \in C$, there exists a covering of the surface $\widetilde{S}_{4\epsilon_5}(\zeta) := \widetilde{u}^{-1}(B_{4\epsilon_5}(\widetilde{u}(\zeta))$ by $2 - 3\chi(C)$ tracts.
\end{lem}

Note that $\chi(C) \leq 0$ by \ref{label:l1}, so the number $2 - 3\chi(C)$ from the lemma is indeed positive. Recall that a tract $\widetilde{C}$ is \emph{rectangular} if it has zero genus, two horizontal boundary components, and two vertical boundary components. The proof of Lemma~\ref{lem:tract_geometry} requires the following technical lemma, which asserts that $\widetilde{S}_{4\epsilon_5}(\zeta)$ cannot intersect both vertical boundary components of a rectangular tract. The proof of this lemma is similar to the proof of \cite[Lemma $4.35$]{FH23}.

\begin{lem}\label{lem:tract_geometry_2}
For any point $\zeta \in C$ and any $k$ such that $\widetilde{C}_k$ is rectangular, the surface $\widetilde{S}_{4\epsilon_5}(\zeta)$ does not intersect both vertical boundary components of $\widetilde{C}_k$. 
\end{lem}

\begin{proof}
Let $\widetilde{C} := \widetilde{C}_k$ denote any rectangular tract. Then $\widetilde{C}$ has two horizontal and two vertical boundary components. Let $L$ denote the bottom horizontal boundary component, defined rigorously as the unique horizontal boundary component contained in $(a \circ \widetilde{u})^{-1}(a_0)$. Then $L$ is a compact interval connecting the two vertical boundary components. Write $\partial L = \zeta_+ - \zeta_-$ where $\zeta_{\pm}$ are distinct points in $\widetilde{C}$. Write $\gamma_{\pm}$ for the vertical boundary components intersecting $L$ at $\zeta_{\pm}$ respectively. Write
$$d := \inf_{\substack{z_+\in\gamma_+\\ z_-\in\gamma_-}} \operatorname{dist}_g(\widetilde{u}(z_+), \widetilde{u}(z_-))$$
for the extrinsic distance between $\gamma_+$ and $\gamma_-$. Since both $\gamma_+$ and $\gamma_-$ have length at most $c_3(a_1 - a_0)$ by Proposition~\ref{prop:tract_decomposition}(e), it follows from the triangle inequality that
$$d \geq \operatorname{dist}_g(\widetilde{u}(\zeta_+), \widetilde{u}(\zeta_-)) - 2c_3(a_1 - a_0).$$

Let $q: [0, T] \to C$ denote the unique unit-speed parameterization of $L$ such that $q(0) = \zeta_-$ and $q_+ = \zeta_+$. Apply \ref{label:l6}, Proposition~\ref{prop:tract_decomposition}(d), and Lemma~\ref{lem:geodesic_distance} to the curve $\widetilde{q} := \widetilde{u} \circ q$ to bound $\operatorname{dist}_g(\widetilde{u}(\zeta_+), \widetilde{u}(\zeta_-))$ from below. We deduce the bound
$$d \geq \epsilon_4/2 - 2c_3(a_1 - a_0).$$ 
The right-hand side is seen to be strictly greater than $8\epsilon_5$ using \ref{label:l4} and the bound $\epsilon_5 \leq 2^{-24}\min(c_3^{-3}, \epsilon_4)$.
We conclude that $\gamma_+$ and $\gamma_-$ have extrinsic distance greater than $8\epsilon_5$ from each other. By the triangle inequality, they cannot both intersect $\widetilde{S}_{4\epsilon_5}(\zeta)$ for any choice of $\zeta \in C$. 
 \end{proof}

Lemma~\ref{lem:tract_geometry} is proved by combining Lemmas~\ref{lem:tract_topology_2} and \ref{lem:tract_geometry_2} with a combinatorial argument. 

\begin{proof}[Proof of Lemma~\ref{lem:tract_geometry}] 
Let $\widehat{C}_1, \ldots, \widehat{C}_D$ denote a minimal-size cover of $\widetilde{S}_{4\epsilon_5}(\zeta)$ by tracts. Our goal is to prove the bound $D \leq 2 - 3\chi(C)$. We assume without loss of generality that $D \geq 2$. The proof will take $4$ steps. 

\noindent\textbf{Step $1$:} Write $Z$ for the number of indices $i$ such that $\widehat{C}_i$ is rectangular. This step observes that $Z \geq D + \chi(C)$. This is a direct consequence of Lemma~\ref{lem:tract_topology_2} and Remark~\ref{rem:tract_topology}. It is important that we assume $D \geq 2$ here, to avoid the degenerate case stated in Lemma~\ref{lem:tract_topology_2}(b). 

\noindent\textbf{Step $2$:} This step constructs a connected graph $G$ as follows. The vertices are $\{1, \ldots, D\}$. For any $i \neq j$, we add an edge between them if the tracts $\widehat{C}_i$ and $\widehat{C}_j$ share a vertical boundary component.  The connected surface $\widetilde{S}_{4\epsilon_5}(\zeta)$ intersects each of the tracts $\widehat{C}_i$ since they form a cover of minimal size. This implies that $\bigcup_{i=1}^d \widehat{C}_i$ is connected, which in turn implies that $G$ is connected. 

\noindent\textbf{Step $3$:} Let $i$ be any vertex such that $\widehat{C}_i$ is rectangular. This step shows that $i$ is a vertex of degree $1$, which is equivalent to the assertion that $\widetilde{S}_{4\epsilon_5}(\zeta)$ intersects exactly one vertical boundary component of $\widehat{C}_i$. This assertion follows from applying Lemma~\ref{lem:tract_geometry_2} to the rectangular tract $\widehat{C}_i$ and the fact that $G$ is connected. 

\noindent\textbf{Step $4$:} This step uses Lemma~\ref{lem:tract_topology_2} and some basic graph theory to prove that $D \leq 2 - 3\chi(C)$. For each $i \in \{1, \ldots, D\}$, let $N_i$ denote the degree of the vertex $i$ and $M_i$ denote the number of vertical boundary components of $\widehat{C}_i$. Note that $N_i \leq M_i$ for any $i$. By Step $3$, we have $N_i = 1$ and $M_i = 2$ when $\widehat{C}_i$ is rectangular, so we get the improved bound $N_i \leq M_i - 1$ in this case. We deduce the inequality 
\begin{equation}\label{eq:tract_topology_3} 2\chi(C) \leq \sum_{i=1}^D (2\chi(\widehat{C}_i) - M_i) \leq 2D - Z - \sum_{i=1}^D N_i.\end{equation}

The first inequality uses \eqref{eq:tract_topology_2} and the assertion, proved in Lemma~\ref{lem:tract_topology_2}, that every term on its right-hand side is $\leq 0$. The second inequality uses the observed bounds for each $M_i$ above and the bound $\chi(\widehat{C}_i) \leq 1$. The last two terms on the right are controlled by $D$ and $\chi(C)$. We proved that $Z \geq D + \chi(C)$ in Step $1$. Since $G$ is connected, it has at least $D - 1$ edges. The sum $\sum_{i=1}^D N_i$ is twice the number of edges of $G$, so we conclude that $\sum_{i=1}^D N_i \geq 2D - 2$. Plug these bounds into \eqref{eq:tract_topology_3} to get an upper bound on $D$:
$$2\chi(C) \leq 2D - (D + \chi(C)) - (2D - 2)\quad\Rightarrow\quad D \leq 2 - 3\chi(C).$$
\end{proof}

\subsubsection{Proof of Proposition~\ref{prop:local_area_bound_technical}} 

We can now explain the proof of the technical bound Proposition~\ref{prop:local_area_bound_technical}, contingent on Proposition~\ref{prop:tract_decomposition}.

\begin{proof}[Proof of Proposition~\ref{prop:local_area_bound_technical}]

Use Proposition~\ref{prop:tract_decomposition} to construct a $(\delta, \epsilon)$-tame perturbation $(u, f)$ and a tract decomposition $\{\widetilde{C}_k\}_{k=1}^N$. Recall that $\widetilde{u}$ denotes the perturbed map and that $\widetilde{\gamma} = \widetilde{u}^*g$ denotes the perturbed metric. The proposition states that $\epsilon > 0$ can be chosen to be arbitrarily small; we will choose $\epsilon$ to be smaller than $\epsilon_5$. This implies that $\operatorname{dist}_g(\widetilde{u}(\zeta'), u(\zeta')) < \epsilon_5$ for any $\zeta' \in C$. It follows from this that $S_{\epsilon_5}(\zeta) \subseteq \widetilde{S}_{4\epsilon_5}(\zeta)$. Using this observation and Lemma~\ref{lem:metric_comparison}, it follows that
$$\operatorname{Area}_\gamma(S_{\epsilon_5}(\zeta)) \leq \operatorname{Area}_{\gamma}(\widetilde{S}_{4\epsilon_5}(\zeta)) \leq 2\operatorname{Area}_{\widetilde{\gamma}}(\widetilde{S}_{4\epsilon_5}(\zeta)).$$

So, to complete the proof it suffices to bound the area of $\widetilde{S}_{4\epsilon_5}(\zeta)$ with respect to the metric $\widetilde{\gamma}$. By Lemma~\ref{lem:tract_topology}, each tract has at most $4 - 3\chi(C)$ horizontal boundary components, and by Proposition~\ref{prop:tract_decomposition}(d) it follows that
$$\int_{\partial_h\widetilde{C}_k\,\cap\,(a \circ \widetilde{u})^{-1}(a_0)} \widetilde{\alpha} \leq 10\epsilon_4(4 - 3\chi(C))$$
for each $k$. The area bound in Proposition~\ref{prop:fh_area_bound} shows for each $k$ the bound
$$\operatorname{Area}_{\gamma}(\widetilde{C}_k) \leq 2\operatorname{Area}_{\widetilde{\gamma}}(\widetilde{C}_k) \leq c_0(1 - \chi(C))$$
where $c_0(\eta, J) > 0$ is stable and $k$-independent. By Lemma~\ref{lem:tract_geometry}, $\widetilde{S}_{4\epsilon_5}(\zeta)$ is covered by $2 - 3\chi(C)$ tracts, and the desired area bound follows. 
\end{proof}

\subsubsection{Proof of tract decomposition}\label{subsec:tract_decomp}

To conclude, we need to provide the promised proof of Proposition~\ref{prop:tract_decomposition}, which will take up the remainder of the paper. As we explained in our earlier outline of our arguments, a large part of the proof repeats arguments found in \cite{FH23}, so we only provide sketches for much of this part.  On the other hand, many estimates from \cite{FH23} and many of the assumptions \ref{label:l1}--\ref{label:l6} are used in the proof, and it is crucial to keep careful account of the relevant constants, so even in the sketched parts of the proof we are very precise about the estimates and assumptions used. 

\begin{proof}[Proof of Proposition~\ref{prop:tract_decomposition}]

The proof of Proposition~\ref{prop:tract_decomposition} is simple when $\cA(u) = 0$, and we begin by explaining this:
in this case, the map $u$ is a branched covering map from $C$ onto $[a_0, a_1] \times \gamma$, where $\gamma$ is a closed orbit of $R_\eta$. Cut up $\gamma$ into intervals $\{\cI_\ell\}_{\ell = 1}^M$ with length in $(3\epsilon_4, 4\epsilon_4)$ (or leave it be if it is shorter than that) such that the segments $[a_0, a_1] \times \{z\}$ do not intersect a critical value of $u$ for any $\ell$ and any endpoint $z \in \partial\cI_\ell$. For each $\ell$, define $\widetilde{C}_\ell := u^{-1}([a_0, a_1] \times \cI_\ell)$. The set $\{\widetilde{C}_\ell\}_{\ell = 1}^M$ is the desired tract decomposition. 

Thus, we can assume $\cA(u) > 0$, which will be a standing assumption for the rest of the proof. Fix $(\eta, J) \in \cD(Y)$ and a compact, connected $J$-holomorphic curve $u: C \to \bR \times Y$ satisfying \ref{label:l1}--\ref{label:l6}. The proof takes $5$ steps. The first four steps closely follow the first five steps in the proof of \cite[Proposition $4.30$]{FH23}. The stable constants $c_2$, $c_3$, $\epsilon_4$ from Lemmas~\ref{lem:strip_estimate},~\ref{lem:modest_flow_lines},~\ref{lem:geodesic_distance}, respectively, and the stable constant $\epsilon_5$ from \eqref{eq:local_area_const} will appear frequently.

\noindent\textbf{Step $1$:} This step fixes a $(\delta, \epsilon)$-tame perturbation $(u, f)$ where $\delta$ and $\epsilon$ satisfy suitable bounds. We require $\delta$ to be smaller than a stable constant depending on the map $u$. We also require $\delta \ll c_2\cA(u)$, which is only possible because we are assuming $\cA(u) > 0$. We then require $\epsilon$ to be smaller than a constant depending on $u$ and $\delta$ and smaller than the constant $\epsilon_5$.  The proof that such a perturbation exists is given in Step $1$ of the proof of \cite[Proposition $4.30$]{FH23}.

\noindent\textbf{Step $2$:} Steps $2$--$4$ will show that for a large measure set of initial conditions $\zeta \in \partial_h^- C$, there exists a solution of the gradient flow equation
\begin{equation} \label{eq:gradient_flow} q: [0,T] \to C,\quad q'(s) = \operatorname{grad}_{\widetilde{\gamma}}(a \circ \widetilde{u})(q(s)),\quad q(0) = \zeta \end{equation}
terminating on $\partial_h^+C$. As in \cite{FH23}, we define some relevant sets:
\begin{equation} \label{eq:gradient_flow_sets}
\begin{split}
\cC &:= \{\zeta' \in \operatorname{Crit}(u)\,|\,r(\zeta') = \delta/2\},\\
\cD &:= \{\zeta' \in \operatorname{Crit}(u)\,|\,r(\zeta') < \delta/2\}.
\end{split}
\end{equation}
The set $\cD$ is a union of small disks of radius $\delta/2$, centered at the critical points of $u$, and the set $\cC$ is the union of the boundaries of these disks. 

Step $2$ and its proof in particular follows Step $2$ of the proof of \cite[Proposition $4.30$]{FH23}.  Its goal is to show that solutions to \eqref{eq:gradient_flow} only pass through $\cD$ for a small $\widetilde{\alpha}$-measure set of initial conditions. Calling this set of initial conditions $\cD^- \subseteq \partial_h^-C$, we can prove the bound
\begin{equation} \label{eq:bad_points_1} \int_{\cD^-} \widetilde{\alpha} \leq 4c_2\cA(u).\end{equation}

Here is an outline of the proof of \eqref{eq:bad_points_1}. We note that any gradient flow line starting from a point in $\cD^-$ must intersect the set $\cC'$. Then, using the gradient flow, we construct a disjoint union of strips such that i) their top boundaries lie in $\cC$, ii) their bottom boundaries lie in $\cD^-$ and iii) the total $\widetilde{\alpha}$-measure of the bottom boundaries is close to that of $\cD^-$. The $\widetilde{\alpha}$-measure of $\cC$ is $\lesssim \delta$, which is by Step $1$ much less than $c_2\cA(u)$. Then \eqref{eq:bad_points_1} follows from the height bound \ref{label:l4} and Lemma~\ref{lem:strip_estimate}. 

\noindent\textbf{Step $3$:} This step follows Step $3$ of \cite[Proposition $4.30$]{FH23} and its proof. The perturbed height function $a \circ \widetilde{u}$ is Morse on $C \,\setminus\,\cD$. It follows that $a \circ \widetilde{u}$ has finitely many critical points in $C \,\setminus\,\cD$ and each one is non-degenerate. For each $k \in \{0, 1, 2\}$, write $\cM_k$ for the set of index-$k$ critical points of $a \circ \widetilde{u}$ in $C\,\setminus\,\cD$. The goal of this step is to show, for each $k$, that the set of initial conditions in $\partial_h^- C$ whose gradient flow lines limit to a point in $\cM_0$, $\cM_1$, or $\cM_2$ is small. It is clear that only a finite set $\mathcal{F}^- \subset \partial_h^-C$ of initial conditions have gradient flow lines limiting to a point in $\cM_0$ or $\cM_1$, so it remains to control the $\widetilde{\alpha}$-measure of the set of initial conditions whose gradient flow lines limit to a point in $\mathcal{M}_2$. 

We denote this set by $\mathcal{E}^- \subset \partial_h^-C$ and assert the bound 
\begin{equation}\label{eq:bad_points_2} |\int_{\cE^-}\widetilde{\alpha}| \leq 4c_2\cA(u).\end{equation}

The proof of \eqref{eq:bad_points_2} is similar in style to the proof of \eqref{eq:bad_points_1}. By the Morse lemma, there exists for each point $z \in \cM_2$ a circle $\mathcal{C}_z$ of arbitrarily small length, such that any gradient flow line limiting to $z$ intersects $\mathcal{C}_z$ exactly once transversely. We choose such circles $\mathcal{C}_z$ such that their total length is at most $c_2\cA(u)$.  Using the gradient flow, we construct a disjoint union of strips such that i) their top boundaries lie in $\cC_z$ for some $z \in \cM_2$, ii) their bottom boundaries lie in $\cE^-$ and iii) the total $\widetilde{\alpha}$-measure of the bottom boundaries is close to that of $\cE^-$. Then \eqref{eq:bad_points_2} follows from the height bound \ref{label:l4} and Lemma~\ref{lem:strip_estimate}. 

\noindent\textbf{Step $4$:} This step follows Steps $4$ and $5$ of \cite[Proposition $4.30$]{FH23} and their proofs. It shows that for each closed interval $\cI \subset \partial_h^- C$ satisfying
$$\int_{\cI} \widetilde{\alpha} \geq ( (a_1 - a_0)^{-1} + 10c_2)\cA(u)$$
there exists a solution $q: [0, T] \to C$ to the equation
$$q'(s) = \operatorname{grad}_{\widetilde{\gamma}}(a \circ \widetilde{u})(q(s))$$
such that $q(0) \in \cI$, $q(T) \in \partial_h^+C$, and 
$$\operatorname{length}_{\widetilde{\gamma}}(q([0,T])) \leq c_3(a_1 - a_0).$$

To prove this claim, we define $\cT \subset \partial_h^- C$ to be the complement of the set $\cD^- \cup \cE^- \cup \cF^-$. For any point $\zeta \in \cT$, there exists a solution $q$ to \eqref{eq:gradient_flow} such that $q(0) = \zeta$ and $q(T) \in \partial_h^+ C$. It follows from \eqref{eq:bad_points_1} and \eqref{eq:bad_points_2} that
$$\int_{\partial_h^- C\,\setminus\,\cT} \widetilde{\alpha} = \int_{\cD^- \cup \cE^- \cup \cF^-} \widetilde{\alpha} \leq 8c_2\cA(u).$$
It follows that there exists a finite set $\{\cT_k\}_{k=1}^N$ of pairwise disjoint intervals, each contained in $\cT$, such that 
\begin{equation} \label{eq:bad_points_3} \sum_{k=1}^N \int_{\cT_k}\widetilde{\alpha} \geq \int_{\partial_h^- C}\widetilde{\alpha} - 10c_2\cA(u).\end{equation}

Now write $\cI' := \cI\,\cap\,\cup_{k=1}^N \cT_k$. It follows from \eqref{eq:bad_points_3} that
$$\int_{\cI'} \widetilde{\alpha} \geq \int_{\cI} \widetilde{\alpha} - 10c_2\cA(u) \geq (a_1 - a_0)^{-1}\cA(u).$$

The strips associated to the intervals $\cT_k \cap \cI'$ satisfy the assumptions of Lemma~\ref{lem:modest_flow_lines}. The only nontrivial assumptions to check are (iii) and (iv). Assumption (iii) follows from \ref{label:l4} and Assumption (iv) follows from the inequality above. Applying Lemma~\ref{lem:modest_flow_lines} produces the desired gradient flow line. 

\noindent\textbf{Step $5$:} This step uses the result of the previous step to complete the proof of the proposition. Unlike the other steps, this step does not have any close counterpart in the proof of \cite[Proposition $4.30$]{FH23}. Choose a finite cover $\{\cL_j\}_{j = 1}^M$ of $\partial_h^- C$ satisfying the following properties: 
\begin{enumerate}
\item Each $\cL_j$ is homeomorphic to a closed interval or a circle.
\item $\int_{\cL_j} \widetilde{\alpha} < 5\epsilon_4$. 
\item If $\cL_j$ is homeomorphic to a closed interval, then $\int_{\cL_j} \widetilde{\alpha} > 3\epsilon_4$. 
\item For any $j \neq j'$, the interiors of $\cL_j$ and $\cL_{j'}$ (relative to $\partial_h^-C$) are disjoint. 
\end{enumerate}

For any $j$ such that $\cL_j$ is homeomorphic to an interval, we define a sub-interval $\widehat{\cL}_j \subset \cL_j$ as follows. Write $\zeta_j^-$ and $\zeta_j^+$ for the left and right endpoints of $\cL_j$ with respect to the orientation defined by $\widetilde{\alpha}$. Fix interior points $\zeta_j^0, \zeta_j^1 \in \cL_j$ such that $\zeta_j^0$ is to the left of $\zeta_j^1$ and the following holds. Write $\cL_j^0$, $\cL_j^1$, and $\cL_j^2$ for the sub-intervals with oriented boundaries $\zeta_j^0 - \zeta_j^-$, $\zeta_j^1 - \zeta_j^0$, and $\zeta_j^+ - \zeta_j^1$, respectively. Then we require
$$\int_{\cL_j^i} \widetilde{\alpha} \in (\epsilon_4, 2\epsilon_4)$$ for each $i \in \{0,1,2\}$, and set $\widehat{\cL}_j  := \cL_j^1$. 

Note that by \ref{label:l3}, \ref{label:l5}, and the bound $c_2 \leq 2^{24}\epsilon_5^{-1}$ (see \eqref{eq:local_area_const}), we have the bound $( (a_1 - a_0)^{-1} + 10c_2)\cA(u) \leq \epsilon_4$. Therefore, the interval $\widehat{\cL}_j$ satisfies the required length lower bound in Step $4$. It follows from Step $4$ that for each $j \in \{1, \ldots, M\}$ such that $\cL_j$ is homeomorphic to an interval, there exists a point $\zeta_j \in \widehat{\cL}_j$ and a gradient flow trajectory $q_j: [0, T_j] \to C$ such that $q_j(0) = \zeta_j$, $q_j(T_j) \in \partial_h^+C$, and the length of $q_j([0, T_j])$ is at most $c_3(a_1 - a_0)$. 

Set
$$\dot{C} := C\,\setminus\,\cup_j q_j([0,T_j])$$
and write $\{\dot{C}_k\}_{\ell=1}^N$ for the connected components of $\dot{C}$. For each $k \in \{1, \ldots, N\}$, the closure $\widetilde{C}_k$ of $\dot{C}_k$ relative to $C$ is a tract. We verify that $\{\widetilde{C}_k\}_{k=1}^N$ satisfy the properties of Proposition~\ref{prop:tract_decomposition}. Proposition~\ref{prop:tract_decomposition}(a--c) are evident from the construction. The upper bound in Proposition~\ref{prop:tract_decomposition}(d) follows from the fact that any $L \in \pi_0(\partial_h^-\widetilde{C}_k)$ is contained in the union of at most two of the sets $\cL_j$. The lower bound in Proposition~\ref{prop:tract_decomposition}(d) follows from the fact that if $L$ is not a circle, then $L$ must contain either $\cL_j^0$ or $\cL_j^2$ for some $j$. Proposition~\ref{prop:tract_decomposition}(e) follows from the fact that for each $k$ and each component $L' \in \pi_0(\partial_v\widetilde{C}_k)$, there exists some $j$ such that $L' = q_j([0,T_j])$, and therefore $L'$ has length at most $c_3(a_1 - a_0)$.

\end{proof}

\appendix
\section{Verifying assumptions}\label{sec:verify_monotonicity} 

This short appendix colllects some elementary arguments verifying that important classes of maps and flows satisfy the assumptions of our main results. In the case of surface maps, we show that Hamiltonian surface diffeomorphisms and rational area-preserving $2$-torus diffeomorphisms are monotone. We start with Hamiltonian diffeomorphisms.  

\begin{lem}\label{lem:hamiltonian_monotone}
Any Hamiltonian diffeomorphism of a closed, oriented surface $\Sigma$ is monotone. 
\end{lem}

Lemma~\ref{lem:hamiltonian_monotone} follows immediately from the next lemma and the easily verified fact that the identity map is monotone. 

\begin{lem}\label{lem:hamiltonian_isotopy_monotone}
Let $\phi$ and $\phi'$ be a pair of Hamiltonian isotopic area-preserving diffeomorphisms of a closed, oriented surface $\Sigma$ equipped with an area form $\omega$. Then $\phi'$ is monotone if and only $\phi$ is. 
\end{lem}

\begin{proof}
Choose a Hamiltonian function $H: \bR/\bZ \times \Sigma \to \bR$ whose time-one Hamiltonian flow is $\phi^{-1}\phi'$. Write $\{\psi^t\}_{t \in \bR}$ for the Hamiltonian flow of $H$. We use this choice to identify the mapping torii of $\phi$ and $\phi'$:
$$f_H: j_\phi \simeq Y_{\phi'},\qquad(t, p) \mapsto (t, (\psi^t)^{-1}(p)).$$

We compute 
$$f_H^*c_1(V_{\phi'}) = c_1(V_\phi),\quad f_H^*[\omega_{\phi'}] = [f_H^*\omega_{\phi'}] = [\omega_\phi + dH \wedge dt] = [\omega_\phi].$$

It follows from this computation that $\phi'$ is monotone if and only if $\phi$ is. 
\end{proof}

Next, we show that any rational area-preserving torus diffeomorphism is monotone. 

\begin{lem}\label{lem:torus_monotone}
Write $\mathbb{T}^2 := (\bR/\bZ)^2$ and let $\omega := dx \wedge dy$ denote the standard area form. Any area-preserving diffeomorphism $\phi: \mathbb{T}^2 \to \mathbb{T}^2$ has $c_1(V_\phi) = 0$. Therefore, if $\phi$ is rational, then it is monotone. 
\end{lem}

\begin{proof}
Fix any area-preserving diffeomorphism $\phi: \mathbb{T}^2 \to \mathbb{T}^2$. For any matrix $A \in \operatorname{SL}(2, \bZ)$, write $\phi_A$ for the torus map $w \mapsto Bw$. Any area-preserving diffeomorphism is isotopic through area-preserving diffeomorphisms to some $\phi_A$. Such an isotopy identifies mapping torii and Chern classes. Therefore, it is sufficient to prove the lemma under the assumption that $\phi = \phi_A$ for some $A \in \operatorname{SL}(2, \bZ)$.

Write $r(A) := \operatorname{rank}(\ker(A - \operatorname{Id}))$. This is an integer between $0$ and $2$, inclusive. We give separate proofs that $c_1(V_{\phi}) = 0$ depending on the value of $r(A)$. If $r(A) = 0$, then it follows from the Mayer--Vietoris sequence that $b_2(Y_\phi) = 1 + r(A) = 1$ and that the second homology group of $Y_\phi$ is generated by a torus fiber. The class $c_1(V_\phi)$ has zero pairing with a torus fiber, so $c_1(V_\phi) = 0$.

Now, assume $r(A) \geq 1$. Then the matrix $A$ fixes some nonzero vector $v \in \bR^2$. Therefore, the differential of $\phi = \phi_A$ fixes the constant vector field $v$ on $\mathbb{T}^2$. This vector field defines a non-vanishing section of the vertical tangent bundle $V_\phi$. We conclude that $c_1(V_\phi) = 0$. 
\end{proof}

Next, we consider flows on $3$-manifolds. We show that the geodesic flow of a Finsler surface is the Reeb flow of a torsion contact form. 

\begin{lem}\label{lem:finsler_torsion}
Let $F$ be a closed Finsler surface. Then there exists a torsion contact form $\alpha$ on the unit tangent bundle $SF$ whose Reeb vector field generates the geodesic flow. 
\end{lem}

\begin{proof}
We recall, following \cite{DGZ17}, how to realize a Finsler geodesic flow as a Reeb flow of a contact form. Let $v: TF \to \bR$ denote the Finsler norm and set $H := v^2/2$. Define a $1$-form $\alpha$ on $TF$ by the local coordinate expression $\sum_{i=1}^2 \partial_{q^i} H\, dp^i$, where $p^i$ and $q^i = \partial_{p^i}$ denote local coordinates on the base and fiber. It is proved in \cite[Section $2$]{DGZ17} that $\alpha$ restricts to a contact form on $SF$ and its Reeb vector field generates the geodesic flow. 

Now, we claim that, when we endow it with a complex structure $J$, the $2$-plane bundle $\xi := \ker(\alpha)$ has torsion Chern class. Write $\ell$ for the real line bundle defined as the tangent bundle along the fibers of $SF$. Note that $\alpha$ vanishes on $\ell$, so $\ell$ is a sub-bundle of $\xi$. It follows that $\xi$ is isomorphic to the complexification of $\ell$. The induced complex conjugation map, fixing $\ell$ and acting by $-1$ on $J\ell$, identifies $\xi$ with its conjugate bundle $\overline{\xi}$. It follows that $c_1(\xi)$ is $2$-torsion:
$$2c_1(\xi) = c_1(\xi) + c_1(\overline{\xi}) = c_1(\xi) - c_1(\xi) = 0.$$ 
\end{proof}

\bibliographystyle{abbrv}
\bibliography{main}

\end{document}